\documentclass[reqno]{amsart}
\RequirePackage[l2tabu, orthodox]{nag}

\usepackage{amssymb} 

\usepackage[style=numeric]{biblatex}
\usepackage{bbm}
\usepackage{caption}
\usepackage{color}
\usepackage{enumitem}
\usepackage{graphicx}
\usepackage{mathrsfs}
\usepackage{lmodern}
\usepackage{microtype}
\usepackage{hyperref} 
\makeatletter
\DeclareFontFamily{OMX}{MnSymbolE}{}
\DeclareSymbolFont{MnLargeSymbols}{OMX}{MnSymbolE}{m}{n}
\SetSymbolFont{MnLargeSymbols}{bold}{OMX}{MnSymbolE}{b}{n}
\DeclareFontShape{OMX}{MnSymbolE}{m}{n}{
    <-6>  MnSymbolE5
   <6-7>  MnSymbolE6
   <7-8>  MnSymbolE7
   <8-9>  MnSymbolE8
   <9-10> MnSymbolE9
  <10-12> MnSymbolE10
  <12->   MnSymbolE12
}{}
\DeclareFontShape{OMX}{MnSymbolE}{b}{n}{
    <-6>  MnSymbolE-Bold5
   <6-7>  MnSymbolE-Bold6
   <7-8>  MnSymbolE-Bold7
   <8-9>  MnSymbolE-Bold8
   <9-10> MnSymbolE-Bold9
  <10-12> MnSymbolE-Bold10
  <12->   MnSymbolE-Bold12
}{}

\let\llangle\@undefined
\let\rrangle\@undefined
\DeclareMathDelimiter{\llangle}{\mathopen}%
                     {MnLargeSymbols}{'164}{MnLargeSymbols}{'164}
\DeclareMathDelimiter{\rrangle}{\mathclose}%
                     {MnLargeSymbols}{'171}{MnLargeSymbols}{'171}
\makeatother

\theoremstyle{plain} 
\newtheorem{theorem}{Theorem}[section]
\newtheorem{corollary}[theorem]{Corollary}
\newtheorem{proposition}[theorem]{Proposition}
\newtheorem{lemma}[theorem]{Lemma} 

\theoremstyle{definition} 
\newtheorem{definition}[theorem]{Definition}

\theoremstyle{remark} 
\newtheorem{remark}[theorem]{Remark}

\definecolor{shadecolor}{rgb}{1,0.8,0.3}

\newcommand{\IGNORE}[1]{}

\newcommand{\A}{\mathcal{A}}
\newcommand{\C}{\mathscr{C}}

\renewcommand{\L}{\mathscr{L}}
\newcommand{\M}{\mathscr{M}}
\newcommand{\N}{\mathbbm{N}}
\renewcommand{\P}{\mathsf{P}}

\newcommand{\R}{\mathbbm{R}}
\newcommand{\T}{\mathrm{T}}
\newcommand{\U}{\mathscr{U}}
\newcommand{\V}{\mathscr{V}}

\newcommand{\Y}{\mathscr{Y}}

\newcommand{\BM}{\boldsymbol{m}}
\newcommand{\BV}{\boldsymbol{v}}
\newcommand{\CB}{\C_\mathrm{b}}
\newcommand{\GS}{\geqslant}
\newcommand{\LS}{\leqslant}

\newcommand{\BAN}{E}
\newcommand{\BOP}{\mathcal{B}}
\newcommand{\BVS}{\mathrm{BV}}
\newcommand{\DOM}{\mathrm{dom}}
\newcommand{\DST}{\displaystyle}

\newcommand{\EPS}{\varepsilon}
\newcommand{\LEB}{\mathcal{L}}
\newcommand{\LIP}{\mathrm{Lip}}
\newcommand{\LOC}{\mathrm{loc}}
\newcommand{\ONE}{\mathbbm{1}}
\newcommand{\RHO}{\varrho}
\newcommand{\SEP}{\mathcal{M}}
\newcommand{\SBB}{\mathbb{S}}

\newcommand{\FLOW}{\boldsymbol{X}}

\newcommand{\VELO}{\boldsymbol{b}}
\newcommand{\WEAK}{\DOTSB\protect\relbar\protect\joinrel\rightharpoonup}

\newcommand{\BOREL}{\mathcal{B}}
\newcommand{\GRAPH}{\mathrm{graph}}

\newcommand{\POWER}{\mathcal{P}}

\newcommand{\ARGMAX}{\mathrm{argmax}}
\newcommand{\LINEAR}{\mathcal{L}}
\newcommand{\SPHERE}{\mathcal{S}}

\DeclareMathOperator{\CONV}{\mathrm{conv}}

\DeclareMathOperator*{\ESUP}{\mathrm{ess\,sup}}
\DeclareMathOperator*{\REAL}{\mathrm{Re}}

\DeclareMathOperator*{\LIMSUP}{\mathrm{ap\,\!\limsup}}

\numberwithin{equation}{section}

\title 
    [Flow solutions of transport equations]
    {Flow solutions of transport equations}

\author
    {Sholeh Karimghasemi}

\address
    {Sholeh Karimghasemi,
     Lehrstuhl f\"{u}r Mathematik (Analysis),
     RWTH Aachen University,
     Templergraben 55,
     D-52062 Aachen, 
     Germany}
\email
    {karimghasemi@instmath.rwth-aachen.de}

\author
    {Siegfried M\"{u}ller}

\address
    {Siegfried M\"{u}ller,
     Institut f\"{u}r Geometrie und Praktische Mathematik,
     RWTH Aachen University,
     Templergraben 55,
     D-52056 Aachen, 
     Germany}
\email
    {mueller@igpm.rwth-aachen.de}

\author
    {Michael Westdickenberg}

\address
    {Michael Westdickenberg,
     Lehrstuhl f\"{u}r Mathematik (Analysis),
     RWTH Aachen University,
     Templergraben 55,
     D-52062 Aachen, 
     Germany}
\email
    {mwest@instmath.rwth-aachen.de}

\date{\today}

  % Set-valued set functions and measures; integration of set-valued functions; measurable selections

\begin{document}

\begin{abstract} 
Under general assumptions on the velocity field, it is possible to construct a
flow that is forward untangled. Once such a flow has been selected, the
associated transport problem is well-posed.
\end{abstract}

\maketitle 
%\tableofcontents

%%%%%%%%%%%%%%%%%%%%%%%%%%%%%%%%%%%%%%%%%%%%%%%%%%%%%%%%%%%%%%%%%%%%%%%%%%%%%%%%
%%%%%%%%%%%%%%%%%%%%%%%%%%%%%%%%%%%%%%%%%%%%%%%%%%%%%%%%%%%%%%%%%%%%%%%%%%%%%%%%
%%%%%%%%%%%%%%%%%%%%%%%%%%%%%%%%%%%%%%%%%%%%%%%%%%%%%%%%%%%%%%%%%%%%%%%%%%%%%%%%
%%%%%%%%%%%%%%%%%%%%%%%%%%%%%%%%%%%%%%%%%%%%%%%%%%%%%%%%%%%%%%%%%%%%%%%%%%%%%%%%
%%%%%%%%%%%%%%%%%%%%%%%%%%%%%%%%%%%%%%%%%%%%%%%%%%%%%%%%%%%%%%%%%%% Introduction

\section{Introduction}

Transport processes are ubiquitous in the natural and engineering sciences.
There are two complementary ways of describing such phenomena: Lagrangian and
Eulerian. The Lagrangian approach tracks the motion of transported quantities
(mass, charge). Its mathematical formulation involves ordinary differential
equations and flow maps. The Eulerian approach is based on the density of the
transported quantity, which  is  a function of time and space. Its evolution is
described by a first-order pde called the continuity equation. The two
approaches are closely related.

\medskip

In this paper, we  consider transport along a velocity field $\VELO$ in a
spatial domain $\Omega \subset \R^d$. The continuity equation for the density
$\RHO \GS 0$ then takes the form
\begin{equation}
	\partial_t \RHO + \nabla\cdot(\RHO\VELO) = 0
	\quad\text{in $(0,T) \times \Omega$.}
\label{E:CONTEQN}
\end{equation}
We assume that initial data $\RHO(0,\cdot) = \bar{\RHO}$ is given with
$\bar\RHO$ a nonnegative Borel measure. Equation~\eqref{E:CONTEQN} must be
interpreted in the weak (distributional) sense; see below. For simplicity, we
will not consider boundary conditions allowing for in/outflow through the
boundary $\partial\Omega$ or time-dependent spatial domains. While these
generalizations are very relevant for many applications, here we just insist
that the transported quantity, described through its density $\RHO$, does not
exit the domain $\Omega$.

On the Lagrangian side, we are instead concerned with solutions to
\begin{equation}
	\dot{\gamma}(t) = \VELO\big( t,\gamma(t) \big),
	\quad
	\gamma(0) = x
\label{E:DIFFEQN}
\end{equation}
for $t \in [0,T]$ and $x \in \Omega$. If \eqref{E:DIFFEQN} can be solved for all
$x$, then we define the flow
\begin{equation}
\begin{gathered}
	\FLOW \colon [0,T]\times \Omega \longrightarrow \R^d,
\\
	\FLOW(t,x) := \gamma_x(t)
	\quad\text{where $\gamma_x$ is the solution of \eqref{E:DIFFEQN}.}
\end{gathered}
\label{E:FLOW}
\end{equation}
The solution of \eqref{E:CONTEQN} can then be recovered through the push-forward
\begin{equation}
	\RHO(t,\cdot) := \FLOW(t,\cdot) \# \bar\RHO
	\quad\text{for all $t\in[0,T]$.}
\label{E:PUSHF}
\end{equation}
This representation may not be unique if the flow is not. Recall that the
push-forward of the Borel measure $\bar\RHO$ under any Borel map $X \colon
\Omega \longrightarrow \R^d$ is defined by
\begin{equation}
    (X\#\bar\RHO)(A) := \bar\RHO\big( X^{-1}(A) \big)
    \quad\text{for all Borel sets $A\subset \R^d$.}
\label{E:PUH}
\end{equation}

The existence of solutions to \eqref{E:DIFFEQN} locally in time can be
established whenever $\VELO$ is continuous in space, by Cauchy-Peano theorem. In
order for $\gamma(t)$ to remain inside of $\Omega$, it is necessary that at the
boundary $\partial \Omega$ the velocity field $\VELO$ points back into $\Omega$.
Moreover, a mild assumption on the growth of $\VELO(t,x)$ as $|x| \to \infty$
ensures that solutions of \eqref{E:DIFFEQN} exist for all times $t$; see
Section~\ref{S:MSP} for details. If the map $x \mapsto \VELO(t,x)$ is Lipschitz
continuous, then the solutions of \eqref{E:DIFFEQN} are unique.

On the other hand, if the velocity field $\VELO$ is not Lipschitz in space, then
existence and uniqueness for \eqref{E:CONTEQN}/\eqref{E:DIFFEQN} are much more
subtle. Some regularity of $\VELO$ is needed to establish that solutions do
exist at all. One possible strategy is to approximate the continuity equation
\eqref{E:CONTEQN} by a parabolic regularization, as is done in
\cite{AmbrosioTrevisan2014}. In order to establish convergence to a solution of
\eqref{E:CONTEQN} (thereby proving existence) one must show that the commutator
$\nabla\cdot((\P_s\RHO)\VELO)-\P_s(\nabla\cdot(\RHO\VELO))$ vanishes strongly as
$s\to 0$, where $\P_s$ denotes a suitable regularizing semigroup. A sufficient
condition for this is Sobolev regularity of $\VELO$; see DiPerna-Lions
\cite{DiPernaLions1989}. One possible strategy for showing uniqueness of
solutions is to establish the \emph{renormalization property}, which means that
if $\RHO$ satisfies the continuity equation \eqref{E:CONTEQN}, then
$\beta(\RHO)$ satisfies 
\[
	\partial_t \beta(\RHO) + \nabla\cdot\big( \beta(\RHO) \VELO \big)
		+\big( \beta(\RHO)\RHO-\beta(\RHO) \big) \nabla\cdot\VELO = 0
	\quad\text{in $(0,T) \times \Omega$,}
\]
for smooth functions $\beta$. This is a consequence of the chain rule if all
quantities are smooth. In general, it requires a more sophisticated argument;
see \cite{AmbrosioTrevisan2014}. Typically we will not distinguish between a
measure that is absolutely con\-tinuous with respect to Lebesgue measure and its
Radon-Nikod\'{y}m derivative, to simplify notation.

Quite often, constructing approximate solutions $\RHO_\EPS$ of \eqref{E:CONTEQN}
also involves approximating the velocity field, e.g., by convolution $\VELO_\EPS
:= \VELO \star \varphi_\EPS$ with some mollifier $\varphi_\EPS$. Since typically
$\RHO_\EPS$ only converges weak* in the sense of measures it raises the issue of
passing to the limit in the approximate ``momentum'' $\RHO_\EPS\VELO_\EPS$.
There are two scenarios: If $\VELO$ is sufficiently smooth and $\VELO_\EPS
\longrightarrow \VELO$ strongly in the $\sup$-norm, for example, then the
velocity can become part of the test function. Alternatively, one can consider
the convergence $\RHO_\EPS\VELO_\EPS \WEAK \BM$ weak* in the sense of measures
first. One then has to argue that $\BM$ can be disintegrated in the form
$\RHO\VELO'$, which is often done by establishing that $\BM$ is absolutely
continuous with respect to $\RHO$ and using Radon-Nikod\'{y}m theorem. In this
case, it can happen that $\VELO' = \VELO$ only outside of a set of measure zero.

The well-posedness result \cite{DiPernaLions1989} was extended by Ambrosio
\cite{Ambrosio2004} to the case with $\VELO$ a function of bounded variation. A
crucial assumption was that the spatial divergence of $\VELO$ be essentially
bounded, which prevents the density $\RHO$ from vanishing (formation of vacuum)
and from concentrating (formation of singular measures). In particular, if
$\bar\RHO$ is absolutely continuous with respect to $d$-dimensional Lebesgue
measure $\LEB^d$, then so is $\RHO(t,\cdot)$ for all times $t$. Alternatively,
instead of imposing an $\L^\infty$-bound on (the negative part of)
$\nabla\cdot\VELO$ one can simply assume that $\VELO$ is such that
\begin{equation}
	C^{-1} \LS \RHO(t,x) \LS C
	\quad\text{for all $(t,x) \in (0,T)\times\Omega$,}
\label{E:NEARIN}
\end{equation}
with constant $C>0$. Assumption \eqref{E:NEARIN} is strictly weaker than
bounding $\nabla\cdot\VELO$; see Remark~3.15 in \cite{Bonicatto2017}. We then
say that the velocity $\VELO$ is \emph{nearly incompressible}.

As a tool for establishing the existence of flows for \eqref{E:DIFFEQN} Ambrosio
introduced what is nowadays known as the \emph{superposition principle}: there
exists a Borel probability measure $\eta$ on the space $\C([0,T]; \R^d)$ of
continuous curves such that 
\[
	\RHO(t,\cdot) = e_t \# \eta
	\quad\text{for all $t \in [0,T]$}
\]
where $e_t(\gamma) := \gamma(t)$ is the evaluation map. Moreover, the measure
$\eta$ is concentrated on the set of curves that are \emph{a.c.\ solutions of
\eqref{E:DIFFEQN}}; see Definition~\ref{D:ACSOL}. In particular, the
superposition principle implies the existence of solutions of \eqref{E:DIFFEQN}
for $\bar\RHO$-a.e.\ $x \in \Omega$. It follows from an abstract decomposition
result for currents (cf.\ \cite{Smirnov1993}) because the vector measure
$(\RHO,\RHO\VELO)$ satisfying \eqref{E:CONTEQN} can be interpreted as a normal
current. It is important to realize that $\RHO$ and $\VELO$ must be compatible
in the sense that the velocity field is tangential to the set where the density
is concentrated; see \cite{BouchitteButtazzoSeppecher1997, AlbertiMarchese2016}.
These results have been generalized to metric measure spaces, making contact to
the theory of Dirichlet forms; see \cites{AmbrosioTrevisan2014,
StepanovTrevisan2017} and the references therein.

The current state of the theory can be summarized as follows: If there is
existence of solutions of \eqref{E:CONTEQN} in a suitable class of functions,
then existence of solutions of \eqref{E:DIFFEQN} for a.e.\ $x\in\Omega$  can be
derived from the superposition principle. If additionally we have uniqueness for
\eqref{E:CONTEQN}, then one can prove the existence of a unique \emph{regular}
Lagrangian flow (which is a flow that preserves absolute continuity of the
density). A sufficient condition for this is Sobolev- or $\BVS$-regularity of
the velocity field $\VELO$ plus near incompressibility. We refer the reader to
\cites{DeLellis2007, AmbrosioCrippa2008, AmbrosioTrevisan2014,
BianchiniBonicatto2017}. The latter reference also contains a proof of
\emph{Bressan's conjecture}, which is a stability result for regular Lagrangian
flows that will be discussed in detail in Section~\ref{SS:SOFM}.

\medskip

The formal adjoint to the continuity equation is the transport equation
\begin{equation}
	\partial_t u + \VELO\cdot\nabla u = 0
	\quad\text{in $(0,T) \times \Omega$,}
\label{E:TREQN}
\end{equation}
with suitable initial/final data. In our opinion, the most natural way to
approach \eqref{E:TREQN} is to reinterpret the equation in the form
\begin{equation}
	\partial_t (\RHO u) + \nabla\cdot(\RHO u \VELO) = 0
	\quad\text{in $(0,T) \times \Omega$}
\label{E:TREQN2}
\end{equation}
where $\RHO$ solves the continuity equation \eqref{E:CONTEQN}. By a solution of
\eqref{E:TREQN} we will therefore mean a function $u$ that is integrable with
respect to the space-time measure 
\[
	\sigma(dx,dt) := \RHO(t,dx) \,dt
\]
and satisfies \eqref{E:TREQN2} in the distributional sense. This is the setting
for the compressible Euler equations of gas dynamics, for example, where the
Eulerian velocity
\[
	\BV(t,\cdot) \in \L^2\big( \Omega, \RHO(t,\cdot) \big)
	\quad\text{for all $t$,}
\]
the square-integrability expressing the fact that the \emph{kinetic energy} is
finite. If the velocity field $\VELO$ is nearly incompressible in the sense of
\eqref{E:NEARIN}, then integrability of $u$ can be assumed with respect to the
$(d+1)$-dimensional Lebesgue measure.

On the other hand, if there exists a flow $\FLOW$ as in \eqref{E:FLOW} and 
if 
\[
	\text{the map $x \mapsto \FLOW(t,x)$ is invertible,}
\]
then the solution $u$ of \eqref{E:TREQN} is trivial: just consider
\[
	u(t,z) := \bar u\big(t, \FLOW(t,\cdot)^{-1}(z) \big)
	\quad\text{for $t \in [0,T]$, $z\in \FLOW(t,\Omega)$}
\]
where $\bar u := u(0,\cdot)$ is some initial data. Indeed the transport equation
\eqref{E:TREQN} simply expresses the fact that $u$ must be constant along the
integral curves of $\VELO$.

More generally, if we are to solve
\begin{equation}
	\partial_t u + \VELO\cdot\nabla u + cu = f
	\quad\text{in $(0,T) \times \Omega$,}
\label{E:TREQN3}
\end{equation}
for suitable functions $c$ and $f$, then the function
\[
	U(t,x) := u\big( t, \FLOW(t,x) \big)
\]
satisfies the new equation $\partial_t U + CU = F$ where
\begin{equation}
	C(t,x) := c\big( t,\FLOW(t,x) \big)
	\quad\text{and}\quad
	F(t,x) := f\big( t,\FLOW(t,x) \big)
\label{E:FUNCS}
\end{equation}
for all $t\in[0,T]$ and $x\in\Omega$. Recall that $\partial_t \FLOW(t,x) =
\VELO(t, \FLOW(t,x))$ since the flow $\FLOW$ is built from solutions of
\eqref{E:DIFFEQN}. This simplifies the problem considerably: instead of solving
a partial differential equation, we only consider a \emph{family of ODEs}.

The problem of solving \eqref{E:TREQN3} therefore decomposes into two parts:
\begin{enumerate}

\item Geometry: Compute the flow $\FLOW$ and the density $\RHO$.

\item Transport: Solve the simplified equation $\partial_t U + CU = F$.

\end{enumerate}
Notice that the two steps are largely independent: The first step only depends
on the velocity field $\VELO$, not on the data for \eqref{E:TREQN3}, for
instance. On the other hand, in the second step the geometry of the transport is
completely obscured; it enters only through the redefined functions
\eqref{E:FUNCS}, which depend on the flow $\FLOW$. We feel that such a clear
separation of issues has several advantages: It allows us to utilize recent
developments in the theory of Lagrangian regular flows for non-smooth
velocities, such as \cite{BianchiniBonicatto2017}. On the other hand, the
resulting simplified transport equations can be solved using classical
variational formulations and the Lax-Milgram theorem (see Section~\ref{S:TE}),
which even provides error estimates. The paper has two parts: 
\begin{itemize}

\item In Section~\ref{S:MSP} we investigate conditions that ensure the existence
of a flow map $\FLOW$. Different from the regularity assumptions outlined above,
which imply the existence of a \emph{unique} regular Lagrangian flow, we will
consider more general (less regular) velocity fields and then \emph{select} a
suitable family of integral curves (see \cites{CardonaKapitanski2017a,
CardonaKapitanski2017b}), from which we can build the flow.

\medskip

\item In Section~\ref{S:TE} we introduce the notion of flow solutions of
\eqref{E:TREQN3} and then use a generalized Lax-Milgram theorem to establish
existence and uniqueness. We discuss possible numerical approximations using the
(discontinuous) Petrov-Galerkin framework, along the lines of
\cite{BroersenDahmenStevenson2018}.

\end{itemize}

In order to ensure existence of solutions for rather general velocity fields
$\VELO$ we will utilize the Filippov approach where the ordinary differential
equation is replaced by a differential inclusion, for a set-valued right-hand
side that captures the local behavior of $\VELO$ in the neighborhoods of points;
see Section~\ref{SS:DI} for details. The use of Filippov solutions in the
context of well-posedness of transport/continuity equations has already been
explored in \cite{PoupaudRascle1997}. In that paper, the authors do not require
any regularity for $\VELO$ beyond $\L^\infty$-boundedness, but they make the
\emph{a priori assumption} that Filippov solutions are unique so that a flow can
be constructed. Once the flow is defined, solutions of the continuity equations
(called \emph{measure solutions}) are given as the push-forward of initial data
under the flow. In this paper we will replace the uniqueness assumption of
\cite{PoupaudRascle1997} by a suitable selection principle.

Let us also mention the paper \cite{Gusev2018} where it is shown that a suitable
flow to the Borel velocity field $\VELO$ exists if and only if the corresponding
continuity equation has a non-negative measure-valued solution. Since $\VELO$
may not be smooth, this flow is typically only Borel measurable, which will also
be the case in our setting. The result of \cite{Gusev2018} uses a measurable
selection theorem, as we will do here. Its starting point is a representation of
the non-negative solution to the continuity equation as a \emph{superposition of
trajectories}, i.e., in terms of a probability measure on the set of integral
curves of $\VELO$, as in \cite{Ambrosio2004}. The Borel flow in \cite{Gusev2018}
is then constructed by selecting suitable curves from this superposition. In
contrast, we work directly with the set of Filippov solutions corresponding to
$\VELO$, without passing through the continuity equation. Our flow has the
important additional property of being \emph{untangled forward in time}, which
means that whenever integral curves meet, they coincide for all later times.
Untangledness plays a central role in \cite{BianchiniBonicatto2017}, as it
enables one to decompose the transport problem into a family of one-dimensional
problems along integral curves. This will also be the basis of our notion of
flow solution in Section~\ref{S:TE}.

%%%%%%%%%%%%%%%%%%%%%%%%%%%%%%%%%%%%%%%%%%%%%%%%%%%%%%%%%%%%%%%%%%%%%%%%%%%%%%%%
%%%%%%%%%%%%%%%%%%%%%%%%%%%%%%%%%%%%%%%%%%%%%%%%%%%%%%%%%%%%%%%%%%%%%%%%%%%%%%%%
%%%%%%%%%%%%%%%%%%%%%%%%%%%%%%%%%%%%%%%%%%%%%%%%%%%%%%%%%%%%%%%%%%%%%%%%%%%%%%%%
%%%%%%%%%%%%%%%%%%%%%%%%%%%%%%%%%%%%%%%%%%%%%%%%%%%%%%%%%%%%%%%%%%%%%%%%%%%%%%%%
%%%%%%%%%%%%%%%%%%%%%%%%%%%%%%%%%%%%%%%%%%%%%%%%%%%%%%%%%%%%%%%%%%%%%%% Notation

\section{Notation}

Let us introduce notation and general results that are used throughout the paper.

%%%%%%%%%%%%%%%%%%%%%%%%%%%%%%%%%%%%%%%%%%%%%%%%%%%%%%%%%%%%%%%%%%%%%%%%%%%%%%%%
%%%%%%%%%%%%%%%%%%%%%%%%%%%%%%%%%%%%%%%%%%%%%%%%%%%%%%%%%%%%%%%% Set-Valued Maps

\subsection{Set-Valued Maps}

Let $\Omega$ and $\BAN$ be two nonempty metric spaces.

%========== DEFINITION
\begin{definition}\label{D:SVM}
We denote by $\POWER(\BAN)$ the collection of subsets of $\BAN$.
\begin{itemize}

\item A map $S \colon \Omega \longrightarrow \POWER(\BAN)$ will be called a
\emph{set-valued map}.

\item The sets $S(x) \subset \BAN$, with $x \in \Omega$, are called the
\emph{values} of $S$, and
\[
	S(K) := \bigcup_{x\in K} S(x)
	\quad\text{for all $K \subset \Omega$}
\]
is the \emph{image} of $K$ under $S$.

\item We call \emph{domain} and \emph{graph} of $S$ the sets
\begin{gather*}
	\DOM(S) := \{ x \in \Omega \colon S(x) \neq \varnothing \},
\\
	\GRAPH(S) := \{ (x,\gamma) \in \Omega\times\BAN \colon 
		x \in \Omega, \gamma \in S(x) \}.
\end{gather*}

\item We call \emph{weak inverse} and \emph{strong inverse} of $S$ the maps
\[
\begin{gathered}
	S^-(A) := \{ x\in\Omega \colon S(x)\cap A \neq \varnothing \}
\\
	S^+(A) := \{ x\in\Omega \colon S(x) \subset A \}
\end{gathered}
	\quad\text{for all $A \subset \BAN$.}
\]

\end{itemize}
\end{definition}

Note that $S^-(\varnothing) = S^+(\varnothing) = \varnothing$ and $S^+(A) 
\subset S^-(A)$ for all $A\subset \BAN$.

%========== DEFINITION
\begin{definition}\label{D:USC}
Let $S \colon \Omega \longrightarrow \POWER(\BAN)$ be a set-valued map with
$\DOM(S) \neq \varnothing$. Then
\begin{itemize}

\item $S$ is \emph{upper semicontinuous} (abbreviated u.s.c.) at $x\in\Omega$ if
for any open set $V\subset E$ with $S(x) \subset V$, there exists a neighborhood
$U \subset \Omega$ of $x$ with
\[
	U \subset S^+(V)
	\quad\Longleftrightarrow\quad
	S(U) \subset V.
\]
We say that $S$ is u.s.c.\ if $S$ is upper semicontinuous at every $x\in
\Omega$.

% Indeed if $U \subset S^+(V)$, then for all $x\in U$ we have $S(x) \subset V$, by
% definition of strong inverse. This implies that $S(U) := \bigcup_{x\in U} S(x)
% \subset V$. Conversely, if $S(U) := \bigcup_{x\in U} S(x) \subset V$, then $S(x)
% \subset V$ for all $x\in U$. Again by definition of strong inverse, we conclude
% that $U\subset S^+(V)$.

\item $S$ is \emph{compact-valued} if $S(x)$ is compact in $\BAN$ for all 
$x \in \Omega$.

\end{itemize}
\end{definition}

%========== LEMMA
\begin{lemma}\label{L:USC}
Consider a set-valued map $S \colon \Omega \longrightarrow \POWER(\BAN)$ such
that $\DOM(S) = \Omega$. 
%(which implies that $S(x) \neq \varnothing$ for all $x\in \Omega$).
The following statements are equivalent:
\begin{enumerate}%[label=(\roman*)]

\item $S$ is u.s.c.

\item $S^+(V)$ is open in $\Omega$ whenever $V\subset\BAN$ is open.

\item $S^-(W)$ is closed in $\Omega$ whenever $W \subset \BAN$ is closed.

\end{enumerate}
If $S$ is compact-valued and u.s.c., and $K\subset\Omega$ is compact, then
$S(K)$ is compact.
\end{lemma}

%========== PROOF
\begin{proof}
We proceed in three steps (see also \cite{Kisielewicz1991}).
\medskip

\textbf{Step~1.} We first observe that
\[
	\Omega\setminus S^\pm(A) = S^\mp(\BAN\setminus A)
	\quad\text{for all $A \subset \BAN$.}
\]
Indeed for all $x \in \Omega$ and $A\subset \BAN$ we have
\begin{align*}
	x \in \Omega\setminus S^-(A)
		& \quad\Longleftrightarrow\quad
			x \not\in S^-(A)
\\
		& \quad\Longleftrightarrow\quad
			S(x)\cap A = \varnothing
\\
		& \quad\Longleftrightarrow\quad
			S(x) \subset \BAN\setminus A
		\quad\Longleftrightarrow\quad
			x \in S^+(\BAN\setminus A).
\end{align*}
%
% Simlarly, we observe that
% %
% \begin{align*}
% 	x \in \Omega\setminus S^+(A)
% 		& \quad\Longleftrightarrow\quad
% 			x \not\in S^+(A) 
% \\
% 		& \quad\Longleftrightarrow\quad
% 			S(x) \not\subset A
% \\
% 		& \quad\Longleftrightarrow\quad
% 			S(x) \cap (\BAN\setminus A) \neq \varnothing
% 		\quad\Longleftrightarrow\quad
% 			x \in S^-(\BAN\setminus A).
% \end{align*}
% %
The statement with $S^-$ and $S^+$ interchanged can be proved analogously.

\medskip

\textbf{Step~2.} Suppose that $S$ is upper semicontinuous and $V \subset \BAN$
open. By definition of strong inverse, for every $x \in S^+(V)$ we have $S(x)
\subset V$. By definition of $u.s.c.$, there exists a neighborhood $U \subset
\Omega$ of $x$ with $U \subset S^+(V)$. Hence $(1)\Longrightarrow (2)$.

To prove that $(2) \Longrightarrow (3)$, consider a closed subset $W \subset
\BAN$. Then
\begin{align*}
	\text{$\BAN\setminus W$ is open}
		& \quad\Longrightarrow\quad
			\text{$S^+(\BAN\setminus W)$ is open}
\\
		& \quad\Longrightarrow\quad
			\text{$\Omega\setminus S^+(\BAN\setminus W)$ is closed in $\Omega$.}
\end{align*}
Because of Step~1, it follows that
\[
	\text{$S^-(W) = S^-\big( \BAN\setminus(\BAN\setminus W) \big)
		= \Omega \setminus S^+(\BAN\setminus W)$ is closed in $\Omega$.}
\]

Finally, we prove that $(3) \Longrightarrow (1)$. For given $x \in \Omega$
consider any open subset $V\subset \BAN$ with $S(x) \subset V$. Then
$\BAN\setminus V$ is closed, and it follows that
\[
	\text{$\Omega\setminus S^+(V) = S^-(\BAN\setminus V)$ 
		is closed in $\Omega$,}
\]
thus $S^+(V)$ is open; see again Step~1. Moreover, we have that
\begin{align*}
	S(x) \cap (\BAN\setminus V) = \varnothing
		& \quad\Longrightarrow\quad
			x \not\in S^-(E\setminus V)
\\
		& \quad\Longrightarrow\quad
			x \in S^+(V),
\end{align*}
by definition of weak inverse. Since $S^+(V)$ is open there is a neighborhood
$U\subset \Omega$ of $x$ with $U \subset S^+(V)$. As $x \in \Omega$ was
arbitrary, we conclude that $S$ is u.s.c.

\medskip

\textbf{Step~3.} Assume that $S$ is u.s.c.\ and compact-valued and that $K
\subset \Omega$ is compact. Let $\{ V_\alpha \colon \alpha \in \Lambda \}$ be an
open covering of $S(K)$. For any given $x \in \Omega$, the set
\[
	\text{$S(x)$ is compact}
	\quad\text{and}\quad 
	S(x) \subset \bigcup_{\alpha\in\Lambda} V_\alpha.
\]
Therefore there exists a subcovering $\{ V_{\alpha_k} \colon k=1,\ldots,
n(\omega) \}$ such that
\[
	S(x) \subset \bigcup_{k=1}^{n(\omega)} V_{\alpha_k} =: W_x
	\quad\text{and}\quad
	\text{$W_x \subset \BAN$ is open.}
\]
Since $S$ is u.s.c.\ we have that $S^+(W_x)$ is open; see Step~2. Then
$\{S^+(W_x) \colon x \in K\}$ is an open covering of the compact set $K$ and
there exist $x_1, \ldots, x_m \in K$ with
\begin{equation}
	K \subset \bigcup_{i=1}^m S^+(W_{x_i}).
\label{E:KSUBSET}
\end{equation}
By definition of strong inverse, we have that $x \in S^+(A)$ implies $S(x)
\subset A$, thus
\[
	S\big( S^+(A) \big)
		= \bigcup_{x \in S^+(A)} S(x)
		\subset A
	\quad\text{for all $A\subset\BAN$.}
\]
Taking the image of either side of \eqref{E:KSUBSET} under $S$ (see
Definition~\ref{D:SVM}), we obtain
\[
	S(K) \subset S\Bigg( \bigcup_{i=1}^m S^+(W_{x_i}) \Bigg)
		= \bigcup_{i=1}^m S\big( S^+(W_{x_i}) \big) 
		\subset \bigcup_{i=1}^m W_{x_i}.
\]
Since every $W_x$ is a finite union of sets $V_\alpha$ we have indeed found a
finite subcovering of $S(K)$ taken from the open covering $\{ V_\alpha \colon
\alpha\in\Lambda \}$. This proves the result.
\end{proof}

%%%%%%%%%%%%%%%%%%%%%%%%%%%%%%%%%%%%%%%%%%%%%%%%%%%%%%%%%%%%%%%%%%%%%%%%%%%%%%%%
%%%%%%%%%%%%%%%%%%%%%%%%%%%%%%%%%%%%%%%%%%%%%%%%%%%%%%%%%%%%%%%%%% Measurability

\subsection{Measurability}

Here we discuss measurability of set-valued maps.

%========== DEFINITION
\begin{definition}\label{D:MEAS}
Suppose that $(A, \A)$ is a measurable space and that $\BAN$ is a topological
space. We say that a set-valued map $S \colon A \longrightarrow \POWER(\BAN)$
is
\begin{itemize}

\item \emph{weakly measurable} if $S^-(V) \in \A$ for each open subset $V
\subset \BAN$;

\item \emph{measurable}  if $S^-(W) \in \A$ for each closed subset $W \subset
\BAN$.

\end{itemize}
A \emph{measurable selector} from $S$ is a measurable function $f \colon A
\longrightarrow \BAN$ such that
\[
	f(x) \in S(x)
	\quad\text{for all $x\in A$.}
\]
\end{definition}

%========== LEMMA
\begin{lemma}\label{L:EQUIMEAS}
Suppose that $(A, \A)$ is a measurable space and $\BAN$ is a metrizable space.
For a set-valued map $S \colon A \longrightarrow \POWER(\BAN)$ we have the
following:
\begin{enumerate}

\item If $S$ is measurable, then $S$ is weakly measurable.

\item If $S$ is compact-valued and weakly measurable, then $S$ is measurable.

\end{enumerate}
\end{lemma}

%========== PROOF
\begin{proof}
We refer the reader to Lemma~18.2 in \cite{AliprantisBorder2006}.
\end{proof}

%========== LEMMA
\begin{lemma}\label{L:INTMEAS}
Suppose that $(A, \A)$ is a measurable space and $\BAN$ a separable metrizable
space. Consider a sequence of weakly measurable set-valued maps $S_n \colon A
\longrightarrow \POWER(\BAN)$ with closed values such that, for each $x \in A$,
there exists a $k\in\N$ with $S_k(x)$ compact. Then the intersection set-valued
map $I \colon A \longrightarrow \POWER(\BAN)$, defined as
\[
	I(x) := \bigcap_{n=1}^\infty S_n(x)
	\quad\text{for all $x \in A$}
\]
is measurable (hence weakly measurable).
\end{lemma}

%========== PROOF
\begin{proof}
We refer the reader to Lemma~18.4 in \cite{AliprantisBorder2006}.
\end{proof}

%========== DEFINITION
\begin{definition}[Polish Space]%\label{D:POLISH}
A topological space $\SBB$ is called \emph{completely metrizable} if there
exists a distance $d$ that is compatible with the topology s.t.\ $(\SBB, d)$ is
complete. A topological space is called \emph{Polish} if it is separable and
completely metrizable.
\end{definition}
		
Equivalently, a topological space is Polish if it has a countable dense subset
and is homeomorphic to a complete metric space. Notice that Polish-ness only
requires the existence of at least one complete distance compatible with the
given topology. There may be other distances that are not complete. The unit
interval $(0,1)$ in $\R$, for instance, which is open in the usual topology
(therefore not complete), is Polish because it is homeomorphic to $\R$ whose
usual metric is complete.
		
% %========== THEOREM
% \begin{theorem}[Kuratowski--Ryll-Nardzewski Selection Theorem]\label{T:KRNST}
% A weakly measurable set-valued map with nonempty closed values from a measurable
% space into a Polish space (see Definition~\ref{D:POLISH}) admits a measurable
% selector.
% \end{theorem}

% %========== PROOF
% \begin{proof}
% We refer the reader to Theorem~18.13 in \cite{AliprantisBorder2006}.
% \end{proof}

%========== THEOREM
\begin{theorem}[Measurable Maximum Theorem]\label{T:MMT}
Let $(A,\A)$ be a measurable space and $\BAN$ a separable metrizable space. Let
$\Gamma \colon A \longrightarrow \POWER(\BAN)$ be a weakly measurable set-valued
map with nonempty compact values, and suppose that $f \colon A \times \BAN
\longrightarrow \R$ is a Carath\'{e}odory function. Define the value function $m
\colon A \longrightarrow \R$ by
\[
	m(x) := \max_{\gamma\in S(x)} f(x,\gamma)
	\quad\text{for all $x \in A$,}
\]
and the set-valued map $\Gamma_* \colon A \longrightarrow \POWER(\BAN)$ of
maximizers by
\[
	\Gamma_*(x) := \{ \gamma\in \Gamma(x) \colon f(x,\gamma) = m(x) \} 
	\quad\text{for all $x \in A$.}
\]
Then 
\begin{enumerate}

\item the value function $m$ is measurable;

\item the $\mathrm{argmax}$ function $\Gamma_*$ has nonempty and compact values;

\item the $\mathrm{argmax}$ function $\Gamma_*$ is measurable and admits a
measurable selector.

\end{enumerate}
\end{theorem}

%========== PROOF
\begin{proof}
We refer the reader to Theorem~18.19 in \cite{AliprantisBorder2006}.
\end{proof}

%%%%%%%%%%%%%%%%%%%%%%%%%%%%%%%%%%%%%%%%%%%%%%%%%%%%%%%%%%%%%%%%%%%%%%%%%%%%%%%%
%%%%%%%%%%%%%%%%%%%%%%%%%%%%%%%%%%%%%%%%%%%%%%%%%%%%%%%%%%%%%%%%%%%%% Uniqueness

\subsection{Uniqueness}

We discuss methods to separate objects in metric spaces.

%========== DEFINITION
\begin{definition}
We say that a Banach space $\BAN$ has the \emph{Radon-Nikod\'{y}m property}
(abbreviated RNP) if the fundamental theorem of calculus holds for $\BAN$-valued
maps: If $f \colon [a,b] \longrightarrow \BAN$ is absolutely continuous, then
there exists a Bochner integrable function $g \colon [a,b]\longrightarrow \BAN$
with the property that
\[
	f(t) = f(a) + \int_a^t g(s) \,ds
	\quad\text{for all $t \in [a,b]$.}
\]
Then $f$ is differentiable for a.e.\ $t \in [a,b]$ with derivative $f' = g$.
\end{definition}

Recall that a function $f \colon [a,b] \longrightarrow \BAN$ is called
\emph{absolutely continuous} if for every $\EPS>0$ there exists $\delta>0$ such
that $\sum_i \| f(b_i)-f(a_i) \| < \EPS$ for every finite collection $\{(a_i,
b_i)\}$ of disjoint intervals in $[a,b]$ with $\sum_i(b_i-a_i) < \delta$. We say
that $f$ is \emph{Lipschitz continuous} if there exists $M$ such that $\|
f(t)-f(s) \| \LS M |t-s|$ for all $s,t \in [a,b]$. Clearly, any Lipschitz
continuous function is absolutely continuous.

The following result is concerned with uniqueness of the inverse Laplace
transform for Banach space-valued functions. It is known as Lerch's theorem. 

%========== LEMMA
\begin{lemma}\label{L:LERCH}
Suppose that $\BAN$ is a Banach space with the Radon-Nikod\'{y}m property. Let
$\LIP_0(\R_+; \BAN)$ be the space of Lipschitz continuous functions $F \colon
\R_+ \longrightarrow \BAN$ such that $F(0) = 0$. Here $\R_+ := [0,\infty)$. If a
function $F \in \LIP_0(\R_+; \BAN)$ satisfies
\begin{equation}
	\int_0^\infty \exp(-\mu_n t) \,dF(t) = 0
	\quad\text{for all $n\in\N$}
\label{E:LST}
\end{equation}
(in the sense of the Riemann-Stieltjes integral), for a sequence of distinct
real or complex numbers $\mu_n$ such that $\REAL\mu_n \GS \mu > 0$ for some
$\mu>0$ and
\begin{equation}
	\sum_{n=1}^N \bigg( 1-\frac{|\mu_n-1|}{|\mu_n+1|} \bigg)
		\longrightarrow \infty
	\quad\text{as $N\to\infty$,}
\label{E:SUMCOND}
\end{equation}
then $F(t) = 0$ for all $t\in\R_+$.
\end{lemma}

%========== REMARK
\begin{remark}\label{R:RNP}
If $\BAN$ has the Radon-Nikod\'{y}m property and $F \in \LIP_0(\R_+; \BAN)$,
then
\[
	\int_0^\infty g(t) \,dF(t) = \int_0^\infty g(t) F'(t) \,dt
\]
for all $g\in \L^1(\R_+)$ continuous. In this case, the Laplace-Stieltjes
transform \eqref{E:LST} reduces to the Laplace transform. Otherwise, it is a
proper generalization. Separable dual spaces and reflexive spaces have the
Radon–Nikod\'{y}m property; see \cite{Bourgin1983}.
\end{remark}

%========== PROOF
\begin{proof}[Proof of Lemma~\ref{L:LERCH}] There exists an isometric
isomorphism between $\LIP_0(\R_+; \BAN)$ and the space of bounded linear maps
from $\L^1(\R_+)$ to $\BAN$ (Riesz-Stieltjes representation): To any $F \in
\LIP_0(\R_+; \BAN)$ associate $T_F \colon \L^1(\R_+) \longrightarrow \BAN$ such
that
\[
	T_Fg := \int_0^\infty g(t) \,dF(t)
	\quad\text{for all $g \in \L^1(\R_+)$ continuous}
\]
and $T_F \ONE_{[0,t]} := F(t)$ for all $t\in \R_+$. By density, the map $T_F$ is
uniquely determined by these assumptions. On the other hand, the family of
functions
\[
	S := \{ \exp(-\lambda_n t) \colon n\in\N \},
\]
with complex numbers $\lambda_n$ as above, is \emph{total} in $\L^1(\R_+)$,
which means precisely that the only bounded linear map on $\L^1(\R_+)$ that
vanishes on $S$ is the zero functional. We refer the reader to Corollary~1.3 in
\cite{BaeumerNeubrander1994} for additional information.
\end{proof}

We now turn to separating classes of functions on topological spaces.

%========== DEFINITION
\begin{definition}
Let $(X,\tau)$ be a topological space and $\SEP$ a collection of $\R$-valued
Borel measurable functions on $X$. We say that $\SEP$ \emph{separates points} if
for any $x,y \in X$ with $x\neq y$ there exists a $g \in \SEP$ with $g(x)\neq
g(y)$. We say that $\SEP$ \emph{strongly separates points} if for any $x\in X$
and any neighborhood $O_x$ of $x$ there exists a finite collection $\{g_1,
\ldots, g_k \} \subset \SEP$ such that $\inf_{y \not\in O_x} \max_{1\LS i\LS k}
|g_i(x)-g_i(y)| > 0$.
\end{definition}

%========== LEMMA
\begin{lemma}\label{L:SSP}
Let $(X,\tau)$ be a topological space with countable basis and suppose that a
subset $\SEP \subset \C(X;\R)$ strongly separates points. Then there exists a
countable collection $\{g_i\}_{i\in\N} \subset \SEP$ that also strongly
separates points. Moreover, this collection can be taken closed under either
multiplication or addition if $\SEP$ is.
\end{lemma}

%========== PROOF
\begin{proof}
We refer the reader to Lemma~2 in \cite{BlountKouritzin2010}.
\end{proof}

%========== THEOREM
\begin{theorem}\label{T:SEPA}
Let $(X,\tau)$ be a topological space and $\SEP$ a countable collection of
$\R$-valued Borel measurable functions of $X$ that is closed under
multiplication and strongly separates points. If $\mu$ is any Borel probability
measure on $\SBB$, then
\[
	\bigg( \int_\SBB g \,d\mu = 0 \quad\text{for all $g\in \SEP$} \bigg)
	\quad
	\Longrightarrow
	\quad
	\mu = 0.
\]
\end{theorem}

%========== PROOF
\begin{proof}
We refer the reader to Theorem~11(c) in \cite{BlountKouritzin2010}.
\end{proof}

%========== REMARK
\begin{remark}\label{R:FAMILY}
We can apply Lemma~\ref{L:SSP} to a Polish space $(\SBB, d)$, with $\SEP$ the
family of Lipschitz continuous functions with bounded support. Notice that the
topology of a Polish space (which is a separable metric space) does have a
countable basis. Moreover, $\SEP$ strongly separates points on $\SBB$ since the
set
\[
	\Big\{ \big( 1-kd(\cdot, y) \big)_+ \colon y\in \SBB, k\in\N \Big\}
\]
belongs to $\SEP$ and strongly separates points. Then there exists a
countable collection $\{ g_i \}_{i\in \N}$ of elements of $\SEP$ that strongly
separates points. It follows that
\[
	\bigg( \int_\SBB g_i \,d\mu = 0 \quad\text{for all $i\in\N$} \bigg)
	\quad
	\Longrightarrow
	\quad
	\mu = 0,
\]
where $\mu$ is any Borel probability measure on $\SBB$; see
Theorem~\ref{T:SEPA}. Our choice of $\SEP$ is motivated by the theory of
continuity equations on metric spaces, where the satisfaction of the transport
ODE is defined by testing against a class of Lipschitz continuous functions. We
refer the reader to \cite{AmbrosioTrevisan2014} for further information.
\end{remark}

%%%%%%%%%%%%%%%%%%%%%%%%%%%%%%%%%%%%%%%%%%%%%%%%%%%%%%%%%%%%%%%%%%%%%%%%%%%%%%%%
%%%%%%%%%%%%%%%%%%%%%%%%%%%%%%%%%%%%%%%%%%%%%%%%%%%%%%%%%%%%%%%%%%%%%%%%%%%%%%%%
%%%%%%%%%%%%%%%%%%%%%%%%%%%%%%%%%%%%%%%%%%%%%%%%%%%%%%%%%%%%%%%%%%%%%%%%%%%%%%%%
%%%%%%%%%%%%%%%%%%%%%%%%%%%%%%%%%%%%%%%%%%%%%%%%%%%%%%%%%%%%%%%%%%%%%%%%%%%%%%%%
%%%%%%%%%%%%%%%%%%%%%%%%%%%%%%%%%%%%%%%%%%%%%%%%%%%%%% Measurable Semi-Processes

\section{Measurable Semi-Processes}\label{S:MSP}

Starting with the seminal work by DiPerna-Lions \cite{DiPernaLions1989}, a lot
of effort has been devoted to identifying successively weaker conditions that
ensure the existence and, in particular, the \emph{uniqueness} of (regular)
flows for \eqref{E:DIFFEQN}. These conditions typically come in the form of
regularity assumptions on the velocity field $\VELO$.

Here we will explore a different approach. We make only minimal assumptions on
the regularity of the velocity field $\VELO$ that ensure existence of solutions
to \eqref{E:DIFFEQN}. Since then uniqueness is typically lost, in order to be
able to anyway define a flow map $\FLOW$, we \emph{select} among all solutions
of \eqref{E:DIFFEQN} a suitable family of integral curves, one for each starting
point $x\in\Omega$. While the resulting flow  $\FLOW$ may not be unique, it will
still have the crucial semi-group property. Using the push-forward formula
\eqref{E:PUSHF} we obtain a solution of the continuity equation
\eqref{E:CONTEQN}. Again we do not expect uniqueness of solutions to
\eqref{E:CONTEQN}. Instead our approach amounts to \emph{selecting} a suitable
one (namely one that has associated to it a flow map with good properties). This
presents an alternative to the regularity-based approach to transport problems.

Notice that the idea of selecting a suitable solution among many possible ones
is at the heart of the theory of hyperbolic conservation laws. Recent
breakthroughs by De Lellis, Sz\'{e}kelyhidi, and others have conclusively
demonstrated that uniqueness cannot be expected for some fundamental physical
models, such as the Euler/Navier-Stokes equations of gas dynamics; see
\cites{DeLellisSzkelyhidi2009, DeLellisSzkelyhidi2010, BuckmasterVicol2019}.
This has reinvigorated the quest for suitable \emph{entropy conditions} that
among all weak solutions of the equations would pick the one with physical
relevance. The scope of this paper is much narrower since we only consider
transport equations. Moreover, our selection principle is different insofar as
the procedure is motivated purely by mathematical considerations.

%%%%%%%%%%%%%%%%%%%%%%%%%%%%%%%%%%%%%%%%%%%%%%%%%%%%%%%%%%%%%%%%%%%%%%%%%%%%%%%%
%%%%%%%%%%%%%%%%%%%%%%%%%%%%%%%%%%%%%%%%%%%%%%%%%%%%%%%% Differential Inclusions

\subsection{Differential Inclusions}
\label{SS:DI}

Let $\Omega \subset \R^d$ be closed. In the following, open/closed subsets of
and neighborhoods in $\Omega$ will always be understood with respect to the
relative topology. Let us first clarify our solution concept for
\eqref{E:DIFFEQN}.

%========== DEFINITION
\begin{definition}\label{D:ACSOL}
An \emph{a.c.\ solution} (also called Carath\'{e}odory solution) of
\eqref{E:DIFFEQN} is an absolutely continuous map $\gamma \colon [0,T]
\longrightarrow \Omega$ with the property that
\[
	\gamma(t) = x + \int_0^t \VELO\big( s,\gamma(s) \big) \,ds
	\quad\text{for all $t \in [0,T]$.}
\]
\end{definition}

By Cauchy-Peano theorem, if the velocity field $\VELO$ is continuous in space,
then a.c.\ solutions of \eqref{E:DIFFEQN} exist for $T$ sufficiently small. But
generally there is no uniqueness. A standard counterexample is the following
initial value problem
\[
	\dot{\gamma} = 2 \gamma^{1/2},
	\quad
	\gamma(0) = 0,
\]
which has \emph{infinitely many} solutions of the form
\[
	\gamma(t) = \begin{cases}
			0 & \text{if $t < c$}
\\
			(t-c)^2 & \text{if $t \GS c$}
		\end{cases}
	\quad\text{for all $t\in[0,T]$,}
\]
for any $c\GS 0$. If $\VELO$ is discontinuous in space, then even existence for
\eqref{E:DIFFEQN} may no longer be given, as the following example demonstrates:
If the velocity field
\[
	\VELO(x) := \begin{cases}
			+1 & \text{for $x \LS 0$}
\\
			-1 & \text{for $x > 0$}
		\end{cases}
\]
then there exists no solution of \eqref{E:DIFFEQN} for initial data $x = 0$.

\medskip

In order to have a robust existence theory at our disposal, instead of the
differential equation \eqref{E:DIFFEQN} we will consider \emph{differential
inclusions} of the form
\begin{equation}
	\dot{\gamma}(t) \in F\big( t,\gamma(t) \big)
	\quad
	\text{for $s\LS t\LS T$,}
	\quad
	\gamma(s) = x
\label{E:DIFFINCL}
\end{equation}
with $(s,x) \in [0,T]\times\Omega$ and $F \colon [s,T] \times \Omega
\longrightarrow \POWER(\R^d)$; see Definition~\ref{D:SVM}.

%========== DEFINITION
\begin{definition}
An \emph{a.c.\ solution} (also called Carath\'{e}odory solution) of
\eqref{E:DIFFINCL} is an absolutely continuous map $\gamma \colon [s,T]
\longrightarrow \Omega$ with the property that
\[
	\dot\gamma(t) \in F\big( t,\gamma(t) \big) 
	\quad\text{for a.e.\ $t \in [s,T]$.}
\]
\end{definition}

In order for a.c.\ solutions of \eqref{E:DIFFINCL} to remain inside of $\Omega$
it is necessary that the velocity field at the boundary does not point into the
complement of $\Omega$.

%========== DEFINITION
\begin{definition}\label{D:TANCONE}
We define the \emph{tangent cone} to $\Omega$ at the point $x$ as
\[
	\T_x \Omega := \bigg\{ v \in \R^d \colon \liminf_{\lambda\to 0+}
		\lambda^{-1} d(x+\lambda v, \Omega) = 0 \bigg\}
	\quad\text{for all $x \in \Omega$.}
\]
Here $d(x,y) := \| x-y \|$ for  $x,y \in \R^d$ is the induced distance.
\end{definition}

We then require that $\VELO(t,x) \in \T_x\Omega$ for all $x \in \partial
\Omega$. Note that $\T_x\Omega = \R^d$ if $x \in \mathring{\Omega}$. We can now
state the main existence result (recall Definitions~\ref{D:USC} and
\ref{D:MEAS}).

%========== THEOREM
\begin{theorem}\label{T:EXISTENCE}
Let $I := [s,T]$. Suppose that a set-valued map $F \colon I \times \Omega
\longrightarrow \POWER(\R^d)$ is given with nonempty closed convex values and
with the following properties:
\begin{enumerate}

\item For all $t\in I$, the map $x\mapsto F(t,x)$ is upper semicontinuous.

\item For all $x\in \Omega$, the map $t\mapsto F(t,x)$ is measurable.

\item For all $(t,x) \in [s,T)\times\Omega$, we have $F(t,x) \cap \T_x\Omega
\neq \varnothing$.

\item There exists a function $c\in \L^1(I)$ such that
\begin{equation}
	\|F(t,x)\| \LS c(t) (1+|x|)
	\quad\text{for all $(t,x) \in I\times \Omega$,}
\label{E:GROWTH}
\end{equation}
where $\|F(t,x)\| := \sup\{ |y| \colon y\in F(t,x) \}$.
\end{enumerate}
Then there exists an a.c.\ solution of \eqref{E:DIFFINCL} for every $x\in
\Omega$.
\end{theorem}

%========== PROOF
\begin{proof}
We refer the reader to Theorem~5.2 in \cite{Deimling1992}.
\end{proof}

%========== REMARK
\begin{remark}\label{R:GRONWALL}
By Gronwall's lemma and assumption~\eqref{E:GROWTH}, we have
\[
	|\gamma(t)| \LS \bigg( |x| + \int_s^t c(r) \,dr \bigg)
		\exp\bigg( \int_s^t c(r) \,dr \bigg)
\]
for all $s\LS t\LS T$. In particular, the solutions of \eqref{E:DIFFINCL} remain
bounded.
\end{remark}

The following topological properties of the solution set of \eqref{E:DIFFINCL}
will be crucial.

%========== THEOREM
\begin{theorem}\label{T:DEIMLING}
For any $x\in \Omega$ let $\Gamma(s,x)$ be the set of a.c.\ solutions of
\eqref{E:DIFFINCL}.
\begin{enumerate}

\item For all $x$ the set $\Gamma(s,x)$ is nonempty and compact.

\item The map $x\mapsto \Gamma(s,x)$ is upper semicontinuous.

\end{enumerate}
\end{theorem}

%========== PROOF
\begin{proof}
We refer the reader to Theorem~7.1 in \cite{Deimling1992}.
\end{proof}

Here $\Gamma(s,x)$ is understood as a subset of $\C([s,T]; \Omega)$, the space
of continuous curves $\gamma \colon [s,T] \longrightarrow \R^d$ with $\gamma(t)
\in \Omega$ for all $t$, endowed with the $\sup$-norm. 
% See page 26 bottom in \cite{Deimling1992}.

\medskip

Let us now return to the transport equation \eqref{E:TREQN3} for a given
velocity field $\VELO$. As outlined in the Introduction, we intend to approach
this problem by factoring out the geometry, i.e., by constructing first a
suitable density $\RHO$ satisfying \eqref{E:CONTEQN} and a flow map $\FLOW$ as
in \eqref{E:FLOW}. In order to be able to use Theorem~\ref{T:EXISTENCE}, we
define
\begin{equation}
	F(t,x) := \bigcap_{\xi \in \SPHERE^{d-1}} \big\{ v\in\R^d \colon
		\xi\cdot v \LS h(t,x,\xi) \big\}
	\quad\text{for $(t,x) \in [0,T] \times \Omega$,}
\label{E:CONVE}
\end{equation}
with essential supporting function $h(t,x,\xi) := \lim_{\delta\to 0}
h_\delta(t,x,\xi)$ and
\begin{align}
	h_\delta(t,x,\xi) & := \ESUP_{y\in B_\delta(x)\cap\Omega}
		\big( \xi\cdot\VELO(t,y) \big)
\label{E:SUPPF}
\\
		& \hphantom{:}= 
			\inf_{\LEB^d(N) = 0} \sup_{y \in B_\delta(x)\setminus N}
				\big( \xi\cdot\VELO(t,y) \big).
\nonumber
\end{align}
By construction, the behavior of $\VELO(t,\cdot)$ on $\LEB^d$-null sets is
irrelevant for the definition of $F(t,x)$. In particular, we may not have
$\VELO(t,x) \in F(t,x)$. We emphasize that the velocity field $\VELO$ is a
concrete function, not an equivalence class modulo null sets.

The essential $\sup$ is monotone with respect to set inclusion: for $A\subset B
\subset \Omega$ and $f \colon \Omega \longrightarrow \bar{\R} := \R \cup
\{\pm\infty\}$ we have that $\ESUP_{x\in A} f(x) \LS \ESUP_{x\in B} f(x)$.
Indeed for any $\EPS>0$ there exists a set $N\subset\Omega$ with $\LEB^d(N) = 0$
such that
\[
	\ESUP_{x\in B} f(x) + \EPS
		\GS \sup_{x\in B\setminus N} f(x)
		\GS \sup_{x\in A\setminus N} f(x)
		\GS	\ESUP_{x\in A} f(x).
\]
Since $\EPS>0$ was arbitrary, the claim follows.

%========== REMARK
\begin{remark}
We remark that \eqref{E:CONVE} is equivalent to the more common definition
\begin{equation}
	F(t,x) := \bigcap_{\delta > 0} \bigcap_{\LEB^d(N) = 0}
		\overline{\CONV} \, \VELO\big( t,B_\delta(x)\setminus N \big)
	\quad\text{for $(t,x) \in [0,T] \times \Omega$,}
\label{E:FILIPPOV}
\end{equation}
which was introduced by Filippov; see \cites{Filippov1960, Filippov1964}.
$F(t,x)$ is the smallest closed convex set, any neighborhood of which contains
the values of $\VELO(t,z)$ for almost all $z$ in some neighborhood of $x$. We
will call \emph{Filippov solution} of the differential equation
\eqref{E:DIFFEQN} any a.c.\ solution of  the differential inclusion
\eqref{E:DIFFINCL} with $F$ defined by \eqref{E:CONVE}.
\end{remark}

%========== PROPOSITION
\begin{proposition}\label{P:COMPATIBLE}
Let $I := [s,T]$ for some $T>0$ and $\Omega \subset \R^d$ closed. Suppose that a
velocity field $\VELO \in \L^\infty_\LOC(I\times \Omega; \R^d)$ is given with
the following properties:
\begin{enumerate}

\item There exists a function $c\in \L^1(I)$ such that
\begin{equation}
	|\VELO(t,x)| \LS c(t) (1+|x|)
	\quad\text{for all $(t,x) \in I\times \Omega$.}
\label{E:VELOB}
\end{equation}

\item For all $(t,x) \in [s,T)\times\partial\Omega$, we have $F(t,x) \cap
\T_x\Omega \neq \varnothing$, where the set-valued function $(t,x) \mapsto
F(t,x)$ is defined in \eqref{E:CONVE}; see also Definition~\ref{D:TANCONE}.
\end{enumerate}
Then the map $F$ satisfies the conditions of Theorem~\ref{T:EXISTENCE}.
\end{proposition}

%========== REMARK
\begin{remark}
From a measure theory point of view it would be very natural to define the
essential supporting function as the \emph{approximate} upper limit
\[
  \tilde{h}(t,x,\xi) := \LIMSUP_{y\to x} \Big( \xi\cdot \VELO(t,y) \Big)
\]
for $(t,x) \in I\times \Omega$ and $\xi \in \SPHERE^{d-1}$; see Section~2.9.12
in \cite{Federer1969}. One can show that
  
\begin{equation}
  \tilde{h}(t,x,\xi) = \lim_{\delta\to 0} \Bigg(
    \inf_{C\in\mathcal{C}} \sup_{y\in B_\delta(x)\cap C} 
    \Big( \xi\cdot \VELO(t,y) \Big)	\Bigg)
\label{E:EDC}
\end{equation}
where $\mathcal{C}$ is  the class of  measurable subsets of $\Omega$ with
density $1$ at $x$; see the argument in \cite{Smallwood1972}. This class
contains in particular sets of the form $\Omega\setminus N$ with $\LEB^d(N) =
0$, which implies that \eqref{E:EDC} is not bigger than the $h(t,x,\xi)$ defined
using \eqref{E:SUPPF}. In order to prove that it can be strictly smaller (and
not u.s.c.), we may use a construction from \cite{Gusev2018} for $d=1$: Consider
a Borel set $P\subset\R$ with the property that 
\begin{equation}
\begin{minipage}[c]{0.7\textwidth}
\centering
for every open interval $J\subset\R$ the sets $J\cap P$ and $J\setminus P$ both 
have positive Lebesgue measure.
\end{minipage}
\label{E:PROPP}
\end{equation}
The existence of such a set can be established arguing as in
\cite{KatznelsonStromberg1974}. We then define $f\colon\Omega \longrightarrow
\R$ by $f(x) := +1$ if $x \in P\cap \Omega$, $f(x) := -1$ otherwise, and set
$\VELO(t,\cdot) := f$ for all $t$. Since $f$ is Borel measurable, it is
approximately continuous almost everywhere; see Theorem~2.9.13 in
\cite{Federer1969}. In particular, we have that $\tilde{h}(t,x,\xi) = \xi f(x)$
for $\xi = \pm 1$, hence $\tilde{F}(t,x) = \{f(x) \}$ is single-valued a.e.,
where $\tilde{F}$ is defined as in \eqref{E:CONVE} with $\tilde{h}$ in place of
$h$. On the other hand, it holds $F(t,x) = [-1,1]$ everywhere because of
\eqref{E:PROPP}. The constant function $\gamma(t) := \alpha$ for $t \in [0,T]$,
where $\alpha \in (-1,1)$, is a Filippov solution corresponding to
$\VELO(t,\cdot) = f$ since $\dot{\gamma}(t) = 0 \in F(t,\gamma(t))$. The
assumption of $\gamma$ being a Carath\'{e}odory solution, however, leads to the
contradiction
\[
  f\big( \gamma(0) \big) 
    = \frac{1}{T} \int_0^T f\big( \gamma(0) \big) \,dt
    = \frac{1}{T} \int_0^T f\big( \gamma(t) \big) \,dt
    = \frac{1}{T} \big( \gamma(T)-\gamma(0) \big) 
    = 0.
\]
In fact, it is shown in \cite{Gusev2018} that the set of Carath\'{e}odory
solutions is empty.
\end{remark}

%========== REMARK
\begin{remark}
The Filippov construction \eqref{E:FILIPPOV} changes the velocity field, which
may be undesired in some applications. On the other hand, the density $\RHO$ and
the velocity field $\VELO$ are often constructed simultaneously from an
approximation, e.g., in the construction of solutions to the compressible Euler
equations.
\end{remark}

%========== REMARK
\begin{remark}
As mentioned in the Introduction, Filippov solutions have already been
considered in \cite{PoupaudRascle1997} in order to construct solutions of the
continuity equation: Under the assumption that Filippov solutions are uniquely
determined, which follows for example from a one-sided Lipschitz condition on
the velocity field $\VELO$ (see \cites{Filippov1988} for details), a unique flow
can be built from Filippov solutions. A \emph{measure solution} of the
continuity equation for $\VELO$ is then defined as the push-forward of the
initial data under this flow. By uniqueness of the flow, uniqueness in the class
of measure solutions follows. We refer the reader to \cite{PoupaudRascle1997}
for details; see also \cite{BouchutJamesMancini2005}.
% There is a stability result for the Filippov flows in
% \cite{BianchiniGloyer2011}, which will be discussed in Section~\ref{SS:SOFM}.

In contrast, our main interest is not on uniqueness, but on existence of flows
for possibly low-regularity velocity fields, for which uniqueness of Filippov
solutions or of the corresponding flows may not be given. As in
\cite{PoupaudRascle1997} we then define solutions of the continuity equation
using the push-forward of initial data under the flow.
\end{remark}

%========== PROOF
\begin{proof}[Proof of Proposition~\ref{P:COMPATIBLE}] The result is well-known;
see \cites{Filippov1988, AubinCellina1984, Deimling1992, Hoermander1997}, for
instance. In these references, the arguments are often only sketched, therefore
we provide a detailed proof for the reader's convenience; see also
\cite{Haller2008}. Note first that $h_\delta(t,x,\xi)$ is well-defined for
\emph{every} $x\in \Omega$ and a.e.\ $t\in I$. Then \eqref{E:VELOB}
implies the bound
\[
	|h_\delta(t,x,\xi)| \LS c(t) (1+|x|+\delta)
	\quad\text{for all $(t,x) \in I\times \Omega$,}
\]
for all $\xi\in\SPHERE^{d-1}$ and $\delta>0$. It is straightforward to check
that the map $\xi \mapsto h_\delta(t,x,\xi)$ is one-homogeneous and convex, and
the same is true for its pointwise limit $h(t,x,\xi)$. It follows that for all
$x\in \Omega$ and a.e.\ $t\in I$ the set $F(t,x)$ is nonempty and convex,
satisfying \eqref{E:GROWTH}. Since $F(t,x)$ is also closed, it is in fact
compact. Recall that in this case, weakly measurable and measurable are
equivalent; see Lemma~\ref{L:EQUIMEAS}.

For the remaining proof, we proceed in two steps.

\medskip

\textbf{Step~1.} The function $x\mapsto \VELO(t,x)$ with $x\in \Omega$ is
measurable for a.e.\ $t\in I$. We claim that that for any such $t$, the map $x
\mapsto h_\delta(t,x,\xi)$ is upper semicontinuous, for every $\xi \in
\SPHERE^{d-1}$ and $\delta>0$. We must prove the following inequality:
\begin{equation}
	\limsup_{y\to x} h_\delta(t,y,\xi) \LS h_\delta(t,x,\xi).
\label{E:ONE}
\end{equation}
We consider any sequence of $y_k\in\Omega$ such that $|y_k-x| \LS 1/k$ and
\[
	\lim_{k\to \infty} h_\delta(t,y_k,\xi)
		= \limsup_{y\to x} h_\delta(t,y,\xi).
\]
Then $B_\delta(y_k) \subset B_{\delta+1/k}(x)$, which implies that
\begin{equation}
	\ESUP_{z\in B_\delta(y_k)\cap\Omega} \big( \xi\cdot\VELO(t,z) \big)
		\LS \ESUP_{z\in B_{\delta+1/k}(x)\cap\Omega} 
			\big( \xi\cdot\VELO(t,z) \big).
\label{E:HUSC}
\end{equation}
Since the balls $B_{\delta+1/k}(x)$ form a nested and decreasing sequence of
sets, the right-hand side of \eqref{E:HUSC} converges to $h_\delta(t,x,\xi)$ as
$k\to\infty$. This proves \eqref{E:ONE}. We now use that the pointwise $\inf$ of
a family of upper semicontinuous functions is upper semicontinuous. This follows
from the fact that for u.s.c.\ functions the preimages of open intervals of the
form $(-\infty,\alpha)$ with $\alpha\in\R$ are open, which implies that the
preimages of such intervals under pointwise infima of u.s.c.\ functions are
unions of open sets, thus open again. Notice that in the definition of
$h(t,x,\xi)$ we can replace $\lim_{\delta\to 0}$ by $\inf_{\delta>0}$ because
$h_\delta(t,x,\xi)$ is nonincreasing as $\delta$ gets smaller. We conclude that
$x \mapsto h(t,x,\xi)$ is upper semicontinuous, for all $\xi\in \SPHERE^{d-1}$
and a.e.\ $t\in I$.
% Upper semicontinuity implies measurability because
% %
% \[
% 	\{x \in \Omega \colon h(t,x,\xi)<\alpha \}
% 	\quad\text{is open for all $\alpha\in\R$.}
% \]
% %

In order to prove that the set-valued map $x\mapsto F(t,x)$ is upper
semicontinuous, for a.e.\ $t\in I$, we first prove that the function has closed
graph. Consider again a sequence of $y_k \in \Omega$ with $y_k \longrightarrow
x$ as $k\to\infty$ and suppose that $v_k \in F(t,y_k)$ for all $k$ with $v_k
\longrightarrow v$. By definition of $F(t,y_k)$ we have $\xi\cdot v_k \LS
h(t,y_k,\xi)$ for all $\xi \in \SPHERE^{d-1}$. By upper semicontinuity of the
essential support function, it follows that
\[
	\xi\cdot v = \lim_{k\to\infty} \xi\cdot y_k
		\LS \limsup_{k\to\infty} h(t,y_k,\xi) \LS h(t,x,\xi).
\]
Since this holds for every $\xi \in \SPHERE^{d-1}$, we conclude that $v\in
F(t,x)$.

It remains to prove that the closed graph property implies upper semicontinuity;
see Corollary~1.1 in \cite{AubinCellina1984}. For simplicity, we define the
set-valued map $S(x) := F(t,x)$ for all $x\in \Omega$ and $t$ as above. Recall
that $S$ has compact values. Let us \emph{fix} $x_0\in \Omega$ for the
following. For any open subset $V \subset \BAN$ with $S(x_0) \subset V$, we must
establish the existence of a neighborhood $U \subset \Omega$ of $x_0$ such that
$S(U) \subset V$. It will be sufficient to consider $S$ restricted to a bounded
neighborhood $U_0\subset \Omega$ of $x_0$, for which $S(U_0)$ is contained in a
compact subset $K \subset E$. This is possible because of \eqref{E:GROWTH}.

We consider the complement $W := K \setminus V$, which is again a compact subset
of $E$. By assumption, $S(x_0) \subset V$. Therefore, for any $y \in W$ we have
$y \not\in S(x_0)$, and thus $(x_0,y) \not\in \GRAPH(S)$. Since $\GRAPH(S)$ is
closed, there exist neighborhoods $N(y)$ of $y$ and $N_y(x_0)$ of $x_0$,
(relatively) open in $E$ and $\Omega$, respectively, such that
\[
	\GRAPH(S) \cap \big( N_y(x_0) \times N(y) \big) = \varnothing
	\quad\text{for all $y\in W$.}
\]
In particular, it follows that $S(x) \cap N(y) = \varnothing$ for every $x\in
N_y(x_0)$. Since $W$ is compact, it can be covered by finitely many $N(y_i)$
where $y_i \in W$ and $i=1\ldots n$. We now define $U := U_0 \cap
\bigcap_{i=1}^n N_{y_i}(x_0)$, which is (relatively) open in $\Omega$. For $x
\in U$ we have $S(x) \subset K = V \cup W$ because $x \in U_0$. On the other
hand, it holds
\[
	S(x) \cap N(y_i) = \varnothing
	\quad\text{for all $i=1\ldots n$,}
\]
because $x \in N_{y_i}(x_0)$. Since the $N(y_i)$ cover $W$, we conclude that
$S(x) \cap W = \varnothing$ for all $x\in U$, thus $S(U) \subset V$. As
$x_0\in\Omega$ was arbitrary, $S$ is upper semicontinuous. 

% We remark in passing that upper semicontinuity of $S$ implies its measurability.
% This can be proved as in Step~2 by using the measurability of $x\mapsto
% h(t,x,\xi)$.

\medskip

\textbf{Step~2.} We now prove measurability of the map $t \mapsto F(t,x)$ for
any \emph{fixed} $x\in \Omega$. Since there is no upper semicontinuity in time,
a new argument different from the one in Step~1 is needed. We first prove that
the essential supporting function
\[
	t \mapsto h(t,x,\xi)
	\quad\text{for $t\in I$ is measurable,}
\]
for every \emph{fixed} $\xi \in \SPHERE^{d-1}$. To simplify notation, for given
$\delta>0$ we write
\[
	A := B_\delta(x),
	\quad
	\mu := \frac{\LEB^d \llcorner \Omega}{|A\cap \Omega|},
	\quad\text{and}\quad
	f_t(x) := \xi \cdot \VELO(t,x).
\]
Notice that $\mu(A) = 1$, by construction. For any $r>0$ the function 
\[
	t \mapsto \int_A \exp\big( r f_t(y) \big) \,\mu(dy) =: \phi(t,r)
	\quad\text{is measurable for any $r>0$;}
\]
see Corollary~3.3.3 in \cite{Bogachev2007}. Recall that the composition of a
measurable function with a continuous one is measurable. Then also the function
\begin{equation}
	t \mapsto \frac{1}{r} \log\big( \phi(t,r) \big) 
	\quad\text{is measurable for any $r>0$.}
\label{E:DELTME}
\end{equation}

We now claim that for a.e.\ $t\in I$ we have
\begin{equation}
	\lim_{r\to\infty} \frac{1}{r} \log\big( \phi(t,r) \big) 
		= \ESUP_{y\in A} f_t(y) =: \alpha_t.
\label{E:ESUP}
\end{equation}
Here the essential $\sup$ is taken with respect to $\mu$. Assume first that
$\alpha_t>0$. Then we estimate $\phi(t,r) \LS \exp(r \alpha_t)$, from which we
get \eqref{E:ESUP} with inequality $\LS$. 

To prove the converse, we need to estimate $\phi(t,r)$ from below. We have
\[
	\phi(t,r) 
		\GS \exp(r\beta) \mu(A\cap\{f_t\GS\beta\})
			+ \int_{A \cap \{f_t<\beta\}} \exp\big( r f_t(y) \big) \,\mu(dy)
		=: \phi_\beta(t,r)
\]
for $0<\beta<\alpha_t$ and $r>0$. The first term in $\phi_\beta(t,r)$ grows
unboundedly as $r\to \infty$ since $\mu(A\cap\{f_t\GS \beta\})>0$ for all
$\beta>0$ strictly less than the $\ESUP$ of $f_t$. Here we have used the
assumption $\alpha_t>0$. The second term in $\phi_\beta(t,r)$ is nonnegative. 

For a.e.\ $y\in A$ and all $-1<h<1$ it holds
\[
	\frac{\exp\big( (r+h) f_t(y) \big)-\exp\big( r f_t(y) \big)}{h}
		= f_t(y) \exp\big( (r+\theta h) f_t(y) \big) 
\]
for some $\theta \in [0,1]$, by the mean value theorem. It follows that
\[
	\bigg| \frac{\exp\big( (r+h) f_t(y) \big)-\exp\big( r f_t(y) \big)}{h} \bigg|
		\LS \|f_t\|_{\L^\infty(A)} \exp\big( (r+1) \|f_t\|_{\L^\infty(A)} \big),	
\]
which is finite for a.e.\ $t$. By dominated convergence, we obtain for $h\to 0$
that
\[
	\partial_r\phi_\beta(t,r) = \beta\exp(r\beta) \mu(A\cap\{f_t\GS\beta\})
		+ \int_{A\cap\{f_t<\beta\}} f_t(y) \exp\big( r f_t(y) \big) \,\mu(dy).
\]
Using de l'Hopital rule and canceling $\exp(r\beta)$, we now obtain the identity
\begin{align}
	& \lim_{r\to\infty} \frac{1}{r} \log\big( \phi_\beta(t,r) \big)
		= \lim_{r\to\infty} \frac{\partial_r\phi_\beta(t,r)}{\phi_\beta(t,r)}
\label{E:FRAC}\\
%	& \qquad = \lim_{r\to\infty}
%		\frac{\DST \beta \exp(r\beta) \mu(A\cap\{f_t\GS\beta\})
%			+ \int_{A\cap\{f_t<\beta\}} f_t(y) \exp\big( r f_t(y) \big) \,\mu(dy)}
%		{\DST \exp(r\beta) \mu(A\cap\{f_t\GS\beta\})
%			+ \int_{A\cap\{f_t<\beta\}} \exp\big( r f_t(y) \big) \,\mu(dy)}
%\nonumber\\
	& \qquad = \lim_{r\to\infty} 
		\frac{\DST \beta \mu(A\cap\{f_t\GS\beta\})
			+ \int_{A\cap\{f_t<\beta\}} f_t(y) \exp\Big( r \big( f_t(y)-\beta 
				\big) \Big) \,\mu(dy)}
		{\DST \mu(A\cap\{f_t\GS\beta\})
			+ \int_{A\cap\{f_t<\beta\}} \exp\Big( r \big( f_t(y)-\beta \big)
				\Big) \,\mu(dy)}.
\nonumber
\end{align}
Note that $|f_t(y) \exp(r(f_t(y)-\beta))| \LS \|f_t\|_{\L^\infty(A)}$ for every
$y\in A\cap\{f_t<\beta\}$, with a similar estimate for the integrand in the
denominator of \eqref{E:FRAC}. The two integrals in \eqref{E:FRAC} converge to
zero as $r\to\infty$, by dominated convergence. Thus
\[
	\lim_{r\to\infty} \frac{1}{t} \log\big( \phi(t,r) \big)
		\GS \beta 
	\quad\text{for all $0<\beta<\alpha_t$.}
\]
This proves the identity \eqref{E:ESUP} in the case $\alpha_t>0$.

If $\alpha_t\LS 0$, then we define the function $g_t := f_t - (\alpha_t-1)$ and
rewrite
\[
	\phi(t,r) = \exp\big( r (\alpha_t-1) \big) 
		\int_A \exp\big( r g_t(y) \big) \,\mu(dy)
	\quad\text{for all $r>0$.}
\]
Recall that $\mu(A)=1$. We have $\ESUP_{x\in A} g_t(y) = 1$, which is positive.
Then
\begin{align*}
	\lim_{r\to\infty} \frac{1}{r} \log\big( \phi(t,r) \big) 
		& = (\alpha_t-1) + \lim_{r\to\infty} \frac{1}{r} \log\bigg(
			\int_A \exp\big( r g_t(y) \big) \,\mu(dy) \bigg)
\\
		& = (\alpha_t-1) + 1 = \alpha_t,
\end{align*}
where we have used identity \eqref{E:ESUP} with $g_t$ in place of $f_t$.

Combining \eqref{E:DELTME} and \eqref{E:ESUP}, we conclude that the map
\begin{equation}
	t \mapsto \ESUP_{y\in B_\delta(x)\cap\Omega} \big( \xi\cdot \VELO(t,y) \big)
		= h_\delta(t,x,\xi)
	\quad\text{for $t\in I$ is measurable}
\label{E:DETH}
\end{equation}
since the pointwise limits of sequences of measurable functions are measurable.
The essential supporting function $t \mapsto h(t,x,\xi)$, which is obtained from
\eqref{E:DETH} by sending $\delta\to 0$, is measurable as well, for all
$x\in\Omega$ and $\xi \in \SPHERE^{d-1}$ fixed. 

For $x\in \Omega$ fixed we define $F_x(t) := F(t,x)$ for all $t\in I$. We must
show that 
\[
	F_x^-(B) = \{ t\in I \colon F_x(t)\cap B \neq \varnothing \}
	\quad\text{is measurable for all $B\subset \R^d$ open.}
\]
We follow Theorem~8.3.1 in \cite{AubinFrankowska2009} and consider first the
case of an open ball $B = B_r(y)$ centered at $y\in \R^d$ with radius $r>0$.
Then $F_x(t)\cap B_r(y)$ is nonempty if and only if the Euclidean distance $d(y,
F_x(t))$ is strictly less than $r$. Thus
\begin{equation}
	F_x^-\big( B_r(y) \big) 
		= \big\{ t\in I \colon d\big( y,F_x(t) \big) < r \big\}.
\label{E:PREIM}
\end{equation}
Since the set $F_x(t)$ is closed and convex, the distance $d(y,F_x(t))$ can be
expressed in terms of $h(t,x,\xi)$ as follows (see Theorem~8.2.14 in
\cite{AubinFrankowska2009}):
\begin{align}
	d\big(y,F_x(t)) 
		& = d\big(0,F_x(t)-y)
\label{E:DISTF}\\
		& = -\inf_{\xi\in\SPHERE^{d-1}} \sup_{a\in F_x(t)} 
			\big( \xi\cdot(a-y) \big)
		= \sup_{\xi\in\SPHERE^{d-1}} \big( \xi\cdot y-h(t,x,\xi) \big).
\nonumber
\end{align}
It suffices to consider in \eqref{E:DISTF} the $\sup$ over a \emph{countable}
set $S$ that is dense in $\SPHERE^{d-1}$. Indeed it is known that a convex
function defined on an open convex set $U$ in a normed vector space is Lipschitz
continuous on any compact subset of $U$; see \cite{RobertsVarberg1974}. If the
$\sup$ in \eqref{E:DISTF} is attained at $\xi \in \SPHERE^{d-1}$, then there
exists a sequence of $\xi_k \in S$ with $\xi_k \longrightarrow \xi$ and
$h(t,x,\xi_k) \longrightarrow h(t,x,\xi)$ as $k\to\infty$, by density of $S$.

Since the function $t \mapsto h(t,x,\xi)$ is measurable for every $x\in\Omega$
and $\xi\in S$ fixed, it follows that the distance \eqref{E:DISTF}, as the
pointwise $\sup$ of countably many measurable functions, is measurable in $t$
for every $y\in\R^d$. This implies that \eqref{E:PREIM} is measurable for any
$B_r(y)$. For a general open set $B\subset \R^d$ we have the identity
\begin{equation}
	F_x^-(B) = \bigcup_{n\in\N} F_x^{-1}\big( B_{r_n}(y_n) \big),
\label{E:FMINB}
\end{equation}
with $(B_{r_n}(y_n))_n$ a countable family of open balls whose union equals $B$.
Thus \eqref{E:FMINB} is measurable for all $B\subset\R^n$ open, which proves
that $F_x$ is measurable.
\end{proof}

%%%%%%%%%%%%%%%%%%%%%%%%%%%%%%%%%%%%%%%%%%%%%%%%%%%%%%%%%%%%%%%%%%%%%%%%%%%%%%%%
%%%%%%%%%%%%%%%%%%%%%%%%%%%%%%%%%%%%%%%%%%%%%%%%%%%%%%%%%% Selection of Flow Map

\subsection{Selection of Flow Map}

Here is the main result of Section~\ref{S:MSP}.

%========== THEOREM
\begin{theorem}\label{T:SELECTION}
Let $I := [0,T]$ for some $T>0$ and $\Omega \subset \R^d$ closed. Suppose that a
velocity $\VELO \in \L^\infty_\LOC(I\times \Omega; \R^d)$ is given with the
properties listed in Proposition~\ref{P:COMPATIBLE}. Then there exists a map
$\FLOW \colon I\times I\times \Omega \longrightarrow \Omega$ such that
\begin{enumerate}

\item For all $0\LS s\LS t\LS T$, the map $x \mapsto \FLOW(t, s, x)$ is Borel
measurable.

\item For all $(s,x) \in I \times \Omega$, the map $t \mapsto \FLOW(t, s, x)$ is
a Filippov solution of 
\begin{equation}
	\dot{\gamma}(t) = \VELO\big( t,\gamma(t) \big)
	\quad
	\text{for $s\LS t\LS T$,}
	\quad
	\gamma(s) = x.
\label{E:DIFFEQN2}
\end{equation}
\item The flow has the semigroup property: for all $x\in \Omega$ it holds
\begin{equation}
	\FLOW(t, r, x) = \FLOW\big( t, s, \FLOW(s, r, x) \big)
	\quad\text{for $0\LS r\LS s\LS t\LS T$.}
\label{E:SEMIGROUP}
\end{equation}
\end{enumerate}
\end{theorem}

%========== REMARK
\begin{remark}
Theorem~\ref{T:SELECTION} does not state that the map $x \mapsto \FLOW(t, s, x)$
is injective for all $0\LS s\LS t\LS T$, nor that flows lines are unique for all
$x\in \Omega$. The semigroup property \eqref{E:SEMIGROUP} does imply, however,
that whenever two characteristics starting a distinct locations in $\Omega$ meet
at some later time, then the integral curves will coincide from that time on.
Indeed suppose that $x_1, x_2 \in \Omega$ and $r_1, r_2, s\in I$ are such that
\[
	\FLOW(s,r_1,x_1) = \FLOW(s,r_2,x_2) =: x_m,
\]
with $s\GS r_1, r_2$. Applying \eqref{E:SEMIGROUP} twice, we observe that
\begin{align*}
	\FLOW(t,r_1,x_1) = \FLOW(t,s,x_m) = \FLOW(t,r_2,x_2)
	\quad\text{for all $s\LS t\LS T$.}
\end{align*}
Using language from \cite{BianchiniBonicatto2017}, we say the flow lines are
\emph{forward untangled}.
\end{remark}

%========== REMARK
\begin{remark}
We emphasize again that while the theory outlined in the Introduction provides
sufficient conditions in terms of regularity of $\VELO$ so that flow lines are
unique, here we make only minimal assumptions on $\VELO$, which cannot rule out
non-uniqueness. But then we \emph{select} integral curves so that we are still
able to define a flow.
\end{remark}

%========== PROOF
\begin{proof}[Proof of Theorem~\ref{T:SELECTION}] The proof follows the approach
developed in \cites{CardonaKapitanski2017a,CardonaKapitanski2017b}. It is based
on a repeated application of the Measurable Maximum Theorem. Therefore we start
by particularizing Theorem~\ref{T:MMT} to the situation at hand: Recall that
$\Omega \subset \R^d$ is closed. We consider the measurable space
\[
	(A, \mathcal{A}) := \big( \Omega, \BOREL(\Omega) \big)
	\quad\text{with $\BOREL(\Omega)$ the Borel $\sigma$-algebra.}
\]
As mentioned before, open/closed sets and neighborhoods in $\Omega$ are
understood with respect to the relative topology. We then consider the Banach
spaces
\[
	\BAN := \C([s,T]; \R^d) 
	\quad\text{endowed with the $\sup$-norm,}
\]
with $s\in I$. Notice that $\C([s,T]; \Omega)$ is a separable metrizable space,
as required by Theorem~\ref{T:MMT}. We define the set-valued map $x \mapsto
\Gamma(s,x)$ with $x\in\Omega$ as
\[
	\Gamma(s,x) := \{ \gamma \in \C([s,T];\R^d) \colon 
		\text{$\gamma$ is a Filippov solution of \eqref{E:DIFFEQN2}} \}.
\]
%
% To simplify notation, we also write $\Gamma(x) := \Gamma(0,x)$ for $x\in
% \Omega$.
Then the conditions of Theorem~\ref{T:MMT} are fulfilled, for any $s\in I$
fixed:
\begin{enumerate}[parsep=1ex, label=(A\theenumi)]

\item \emph{For all $x\in \Omega$ the set $\Gamma(s,x)$ is nonempty and compact.}
\\[1ex]
This follows from Theorem~\ref{T:DEIMLING} and Proposition~\ref{P:COMPATIBLE},
which checks that the assumptions of Theorem~\ref{T:EXISTENCE} are satisfied.
Compactness (with respect to the $\sup$-norm) follows from Arzel\`{a}-Ascoli
theorem and the fact that solutions of \eqref{E:DIFFEQN2} are Lipschitz
continuous since the velocity $\VELO$ is locally bounded.

\item \emph{The map $x\mapsto \Gamma(s,x)$ for $x\in \Omega$ is weakly (Borel) 
measurable.}
\\[1ex]
By Theorem~\ref{T:DEIMLING}, the map $x\mapsto \Gamma(s,x)$ is upper
semicontinuous, hence
\[
	\text{$\Gamma(s,\cdot)^-(W)$ is closed in $\Omega$ whenever $W \subset 
		\C([s,T]; \R^d)$ is closed;} 
\]
recall the equivalent definitions of upper semicontinuity in Lemma~\ref{L:USC}.
This implies that the map is measurable with respect to the Borel algebra
$\BOREL(\Omega)$, which is generated by the closed subsets of $\Omega$; see
Definition~\ref{D:MEAS}. On the other hand, measurability implies weak
measurability; see Lemma~\ref{L:EQUIMEAS}.
\end{enumerate}

We will apply Theorem~\ref{T:MMT} to functionals of the form
\begin{equation}
	f_{\lambda,\varphi}(\gamma) 
		:= \int_s^T \exp(-\lambda t) \varphi\big( \gamma(t) \big) \,dt
	\quad\text{for all $\gamma \in \C([s,T];\R^d)$}
\label{E:FLV}
\end{equation}
with $\lambda>0$ and $\varphi \in \CB(\Omega)$, for any $s\in[0,T]$. The
functional $f_{\lambda,\varphi}$ is well-defined. It is linear and continuous
with respect to the topology induced by the $\sup$-norm. Thus
$f_{\lambda,\varphi}$ is, in fact, a Carath\'{e}odory function $\Omega \times
\BAN \longrightarrow \R$ (which does not explicitly depend on the position $x\in
\Omega$). Applying Theorem~\ref{T:MMT}, we obtain that the set-valued map
$\Gamma_{\lambda,\varphi} \colon I\times\Omega \longrightarrow \POWER(\BAN)$ of
maximizers of $f_{\lambda,\varphi}$, defined by
\begin{equation}
	\Gamma_{\lambda,\varphi}(s,x) := \ARGMAX\Big\{ 
		f_{\lambda, \varphi}(\gamma) \colon \gamma \in \Gamma(s,x) \Big\}
	\quad\text{for $(s,x) \in I \times \Omega$,}
\label{E:GLV}
\end{equation}
again has the following properties, for any $s\in I$ fixed:
\begin{enumerate}[parsep=1ex]

\item \emph{For all $x \in \Omega$ the set $\Gamma_{\lambda,\varphi}(s,x)$ is
nonempty and compact.}
	
\item \emph{The map $x\mapsto \Gamma_{\lambda,\varphi}(x)$ for $x\in \Omega$ is
weakly measurable.}

\end{enumerate}
In particular, we can now repeat the procedure above, substituting the
maximizers-valued function $\Gamma_{\lambda,\varphi}$ for $\Gamma$, with a
different choice of $\lambda'>0$ and $\varphi' \in \CB(\Omega)$. Notice that
since $\Gamma(s,x)$ is compact and $f_{\lambda,\varphi}$ continuous, the $\sup$
of $f_{\lambda,\varphi}$ is attained, which explains why the set of maximizers
$\Gamma_{\lambda,\varphi}(s,x)$ is nonempty for all $(s,x)$. It is compact
because it is a closed (by continuity of $f_{\lambda, \varphi}$) subset of the
compact set $\Gamma(s,x)$.

\medskip

The approach above does not directly use that the family of curves consists of
solutions of the differential equation \eqref{E:DIFFEQN2}. We now highlight two
crucial properties of these curves and show that they are preserved under the
maximization:
\begin{enumerate}[parsep=1ex, label=(A\theenumi)]
\setcounter{enumi}{2}

\item[(A3a)] \emph{For any $0\LS r\LS s\LS T$ and any $\gamma \in \Gamma(r,x)$
with $x\in \Omega$, we have that}
\[
	\gamma|_{[s,T]} \in \Gamma\big( s, \gamma(s) \big).
\]
%
% \\[1ex]
Indeed, any flow line starting at time $r$ at the position $x$ is described by a
curve in $\C([r,T]; \Omega)$ that is a Filippov solution of the differential
equation on that time interval. That is, the function is absolutely continuous,
thus differentiable a.e., and satisfies a differential inclusion for almost
every time. This is preserved when restricted to a smaller time interval
$[s,T]$, and that part of the curve then belongs to the funnel of flow lines
that start at time $s$ at position $\gamma(s)$. This follows immediately from
the definition.

\item[(A3b)] \emph{For any $0\LS r\LS s\LS T$, any $\gamma\in \Gamma(r,x)$ with
$x\in\Omega$, and any $\eta \in \Gamma(s,\gamma(s))$, we have that the spliced
curve $\gamma\bowtie\eta$ belongs to $\Gamma(r,x)$, where}
\begin{equation}
	\gamma\bowtie\eta(t) := \begin{cases}
		\gamma(t) & \text{for $r\LS t\LS s$,}
\\
		\eta(t) & \text{for $s\LS t\LS T$.}
	\end{cases}
\label{E:BOWTIE}
\end{equation}
%
% \\[1ex]
The curves $\gamma$ and $\eta$ are absolutely continuous on $[r,T]$ and $[s,T]$,
respectively. Restricting $\gamma$ to $[r,s]$ and joining it with $\eta$ at time
$s$, we preserve the absolute continuity because $\gamma(s) = \eta(s)$, by
construction. Therefore $\gamma\bowtie\eta$ is again differentiable a.e.\ and
satisfies the required differential inclusion for almost all times in $[r,T]$,
which proves the claim. Note that if there was a unique solution starting at
time $s$ at location $\gamma(s)$, then $\eta$ would have to coincide with
$\gamma$ on the time interval $[s,T]$, thus $\gamma\bowtie\eta = \gamma$. The
splicing of curves is possible because our solution concept requires only that
the differential inclusion/equation  be satisfied almost everywhere in time.
\end{enumerate}

We claim that (A3a)/(A3b) hold with $\Gamma$ replaced by $\Gamma_{\lambda,
\varphi}$; recall \eqref{E:GLV} and \eqref{E:FLV}. Consider a curve $\gamma \in
\Gamma_{\lambda, \varphi}(r,x)$ with $x\in \Omega$. We must show that
$\gamma|_{[s,T]}$ maximizes the functional $f_{\lambda, \varphi}$ over the set
$\Gamma(s, \gamma(s))$, for any $0\LS r\LS s\LS T$. Note that this is not
obvious since the maximization problems for different $s$ are defined over
different sets, so the corresponding maximizers need not be related. We know,
however, that for every $\eta \in \Gamma(s,\gamma(s))$, the spliced curve
$\gamma\bowtie\eta$ belongs to $\Gamma(r,x)$ because of (A3b). Since $\gamma$
maximizes $f_{\lambda, \varphi}$ over the set $\Gamma(r,x)$, we have
\begin{align}
	\int_r^T \exp(-\lambda t) \varphi\big( \gamma(t) \big) \,dt
		= f_{\lambda, \varphi}(\gamma)
		& \GS f_{\lambda, \varphi}(\gamma\bowtie\eta)
\label{E:INQQ}\\
		& = \int_r^T \exp(-\lambda t) \varphi( \gamma\bowtie\eta(t) 
			\big) \,dt
\nonumber
\end{align}
By definition~\eqref{E:BOWTIE}, the second integral can be decomposed as
\[
	\int_r^T \exp(-\lambda t) \varphi( \gamma\bowtie\eta(t) \big) \,dt
		= \int_r^s \exp(-\lambda t) \varphi( \gamma(t) \big) \,dt
		+ \int_s^T \exp(-\lambda t) \varphi( \eta(t) \big) \,dt.
\]
Decomposing the first integral in \eqref{E:INQQ} and simplifying, we find that
\begin{align*}
	f_{\lambda, \varphi}\big( \gamma|_{[s,T]} \big)
		= \int_s^T \exp(-\lambda t) \varphi\big( \gamma(t) \big) \,dt
		& \GS \int_s^T \exp(-\lambda t) \varphi( \eta(t) \big) \,dt
\\
		& \vphantom{\int}
			= f_{\lambda, \varphi}(\eta)
	\quad\text{for all $\eta\in\Gamma\big( s,\gamma(s) \big)$,}
\end{align*}
which proves that indeed $\gamma|_{[s,T]} \in \Gamma_{\lambda, \varphi}(s,
\gamma(s))$. This is statement (A3a).

Consider now $\gamma \in \Gamma_{\lambda, \varphi}(r,x)$ and $\eta \in
\Gamma_{\lambda, \varphi}(s, \gamma(s))$ for some $s\in I$. We must show that
the spliced curve $\gamma\bowtie\eta$ maximizes $f_{\lambda, \varphi}$ over
$\Gamma(r, x)$. Since we have just proved that $\gamma_{[s,T]} \in
\Gamma_{\lambda, \varphi}(s,\gamma(x))$ and by choice of $\eta$, we have the
following identity:
\[
	\int_s^T \exp(-\lambda t) \varphi( \gamma(t) \big) \,dt
		= \max_{\zeta \in \Gamma(s, \gamma(s))} f_{\lambda, \varphi}(\zeta)
		= \int_s^T \exp(-\lambda t) \varphi( \eta(t) \big) \,dt.
\]
Decomposing the integrals again, we obtain
\begin{align*}
	f_{\lambda, \varphi}(\gamma\bowtie\eta)
		%\int_r^T \exp(-\lambda t) \varphi\big( \gamma\bowtie\eta(t) \big) \,dt
		&= \int_r^s \exp(-\lambda t) \varphi( \gamma(t) \big) \,dt
			+ \int_s^T \exp(-\lambda t) \varphi( \eta(t) \big) \,dt
\\
		&= \int_r^s \exp(-\lambda t) \varphi( \gamma(t) \big) \,dt
			+ \int_s^T \exp(-\lambda t) \varphi( \gamma(t) \big) \,dt
		= f_{\lambda, \varphi}(\gamma).
\end{align*}
Since $\gamma$ maximizes $f_{\lambda, \varphi}$ over $\Gamma(r,x)$, so does
$\gamma\bowtie\eta$. This is statement (A3b).

To finish the proof, we proceed in three steps.

\medskip

\textbf{Step~1.} Let $\{ g_i \}_{i\in\N}$ be a countable collection of Lipschitz
continuous functions of bounded support that are closed under multiplication and
strongly separates points on the separable metric space $\Omega\subset\R^d$, endowed with
the Euclidean distance; see Remark~\ref{R:FAMILY}. Consider also a sequence
$(\mu_n)_{n\in\N}$ of distinct complex numbers $\mu_n$ satisfying condition
\eqref{E:SUMCOND}. We now fix some enumeration
\[
	(\lambda_k, \varphi_k) := (\mu_{n_k}, g_{i_k})
\]
of all pairs $(\mu_n, g_i)$. We define recursively the sets $\Gamma_0(s,x) :=
\Gamma(s,x)$ and
\[
	\Gamma_k(s,x) := \ARGMAX\Big\{ f_{\lambda_k, \varphi_k}(\gamma) \colon
		\gamma \in \Gamma_{k-1}(s,x) \Big\}
	\quad\text{with $(s,x) \in I\times\Omega$,}
\]
for all $k\in\N$. Finally, we define the intersections $\Gamma_*(s,x) :=
\bigcap_{k\in\N} \Gamma_k(s,x)$. 

As explained above, all $\Gamma_k$ have the properties (A1)-(A3b). Since
$\C([s,T];\R^d)$ is a Banach space (complete) with respect to uniform
convergence, and since the sets $\Gamma_k(s,x)$ are decreasing, nonempty, and
compact, the intersection $\Gamma_*(s,x)$ is again nonempty (and compact), for
every $(s,x) \in I\times\Omega$. Because of Lemma~\ref{L:INTMEAS}, we conclude
that for all $s\in I$, the map $x\mapsto \Gamma_*(s,x)$ is (Borel) measurable.

\medskip

\textbf{Step~2.} We claim that for each $(s,x) \in I\times \Omega$, the set
$\Gamma_*(s,x)$ contains precisely one element, i.e., one curve $\gamma \in
\C([s,T];\R^d)$ that is a Filippov solution of \eqref{E:DIFFEQN2}. Indeed fix
any $(s,x)$ and consider $\gamma, \eta \in \Gamma_*(s,x)$. By construction, we
have
\[
	f_{\mu_n, g_i}(\gamma) = f_{\mu_n, g_i}(\eta)
	\quad\text{for all $i,n \in \N$.}
\]
Defining $F_{\gamma, i}(t) := \int_0^t g_i(\gamma(r)) \,dr$ for $t\GS 0$
(extending $g_i(\gamma)$ by zero outside $[s,T]$), we obtain a function in
$\LIP_0(\R_+; \R)$. Notice that $\gamma$ remains bounded (see
Remark~\ref{R:GRONWALL}) and $g_i$ has bounded support. We define $F_{\eta, i}$
in a similar way. Then
\[
	\int_0^\infty \exp(-\mu_n t) \,d(F_{\gamma,i}-F_{\eta,i})(t)
		= \int_s^T \exp(-\mu_n t) \Big( g_i\big( \gamma(t) \big) 
			-g_i\big( \eta(t) \big) \Big) \,dt 
		= 0
\]
for all $n\in\N$, for any $i\in\N$; see Remark~\ref{R:RNP}. Applying
Lemma~\ref{L:LERCH}, we get
\[
	F_{\gamma,i}(t) = F_{\eta,i}(t)
	\quad\text{for all $t\GS 0$,}
\]
which implies that $g_i(\gamma(t)) = g_i(\eta(t))$ for a.e.\ $t\in[s,T]$, by
differentiation. But since both $\gamma$ and $\eta$ are continuous on $[s,T]$,
and $g_i$ is continuous as well, we get
\[
	g_i\big( \gamma(t) \big) = g_i\big( \eta(t) \big)
	\quad\text{for all $t\in[s,T]$ and all $i\in\N$.}
\]
The collection $\{g_i\}_{i\in\N}$ strongly separates points, by assumption.
Hence $\gamma(t) = \eta(t)$ for all $t\in[s,T]$. This proves that
$\Gamma_*(s,x)$ contains only one curve.

\medskip

\textbf{Step~3.} For all $(r,x) \in \Omega$, we now define the flow
\[
	\FLOW(t,r,x) := \gamma(t)
	\quad\text{for all $t\in[r,T]$,}
\]
where $\gamma$ is the unique curve in $\Gamma_*(r,x)$. Then $t\mapsto
\FLOW(t,r,x)$ is a Filippov solution of \eqref{E:DIFFEQN2}, by construction.
Moreover, since the map $x \mapsto \FLOW(t,r,x)$ is the composition of the
(Borel) measurable set-valued function $x\mapsto \Gamma_*(r,x)$ with the
evaluation
\[
	e_t \colon \C([r,T];\R^d) \longrightarrow \Omega
\]
defined by $e_t(\gamma) := \gamma(t)$ for all $t\in[r,T]$, which is continuous,
it is Borel measurable. It only remains to establish the semigroup property
\eqref{E:SEMIGROUP}. Notice that if $\gamma \in \Gamma_*(r,x)$ for some $x \in
\Omega$, then $\gamma \in \Gamma_k(r,x)$ for all $k\in\N$, by definition of
$\Gamma_*(r,x)$. Then
\[
	\gamma|_{[s,T]} \in \Gamma_k\big( s,\gamma(s) \big)
	\quad\text{for all $0\LS r\LS s\LS T$ and $k\in\N$}
\]
because of property (A3a), which implies that $\gamma|_{[s,T]} \in
\Gamma_*(s,\gamma(s))$. But $\Gamma_*(s,\gamma(s))$ contains exactly one curve:
the one we have used to define the map $t\mapsto\FLOW(t,s,\gamma(s))$, which
must therefore coincide with $\gamma|_{[s,T]}$,i.e., with $\gamma \in
\Gamma_*(r,x)$ restricted to $[s,T]$. Since $\FLOW(t,r,x) = \gamma(t)$ for all
$0\LS r\LS t\LS T$ (in particular, for $t=s$), we get
\[
	\FLOW(t,r,x) = \gamma(t) = \FLOW\big( t,s,\gamma(s) \big)
		= \FLOW\big( t,s, \FLOW(s,r,x) \big)
\]
for all $0\LS r\LS s\LS t\LS T$ and $x\in\Omega$. This completes the proof.
\end{proof}

%========== COROLLARY
\begin{corollary}
Let $I := [0,T]$ for some $T>0$ and $\Omega \subset \R^d$ closed. Suppose that a
velocity field $\VELO \in \L^\infty_\LOC(I\times \Omega; \R^d)$ as in
Theorem~\ref{T:SELECTION} is given, with associated flow map $\FLOW$ as defined
there. For any finite Borel measure $\bar{\RHO}\GS 0$ on $\Omega$, we define
\begin{equation}
	\RHO(t,\cdot) := \FLOW(t,0,\cdot)\#\bar{\RHO}
	\quad\text{for all $t\in I$,}
\label{E:RHOT}
\end{equation}
where $\#$ denotes the push-forward operator; see \eqref{E:PUH}. Then $\RHO$ is
a solution of the continuity equation \eqref{E:CONTEQN} in distributional sense,
with initial data $\RHO(0,\cdot) = \bar{\RHO}$.
\end{corollary}

%========== PROOF
\begin{proof}
This follows from standard arguments; see \cite{DeLellis2007}, for instance.
Note that \eqref{E:RHOT} is well-defined since the map $x\mapsto \FLOW(t,0,x)$
is Borel measurable, for all $t\in I$.
\end{proof}

%========== REMARK
\begin{remark}
After possibly modifying it on a set of measure zero, we may assume that the map
$t\mapsto \RHO(t,\cdot)$ for $t\in I$ is continuous with respect to weak*
convergence of measures; see Lemma~3.7 in \cite{DeLellis2007}. If the measure
$\RHO(t,\cdot)$ is absolutely continuous with respect to $\LEB^d$, then we use
the same symbol also for its Radon-Nikod\'{y}m density.
\end{remark}

%%%%%%%%%%%%%%%%%%%%%%%%%%%%%%%%%%%%%%%%%%%%%%%%%%%%%%%%%%%%%%%%%%%%%%%%%%%%%%%%
%%%%%%%%%%%%%%%%%%%%%%%%%%%%%%%%%%%%%%%%%%%%%%%%%%%%%%%%%%%%%%%%%%%%%%%%%%%%%%%%
%%%%%%%%%%%%%%%%%%%%%%%%%%%%%%%%%%%%%%%%%%%%%%%%%%%%%%%%%%%%%%%%%%%%%%%%%%%%%%%%
%%%%%%%%%%%%%%%%%%%%%%%%%%%%%%%%%%%%%%%%%%%%%%%%%%%%%%%%%%%%%%%%%%%%%%%%%%%%%%%%
%%%%%%%%%%%%%%%%%%%%%%%%%%%%%%%%%%%%%%%%%%%%%%%%%%%%%%%%%%%% Transport Equations

\section{Transport Equations}\label{S:TE}

In this section, we first consider a very simple transport equation with
constant velocity field and investigate the well-posedness of its variational
formulation. Then we will introduce the notion of flow solution to the transport
equation \eqref{E:TREQN3} and establish their existence and uniqueness. Finally,
we will discuss stability properties and possible numerical approximations using
the Petrov-Galerkin framework.

We start with the simple transport equation
\begin{equation}
\begin{alignedat}{2}
	\partial_t U + CU & = F
		&& \quad\text{in $Q := (0,T)\times D$,}
\\
	U(0,\cdot) & = \bar{U}
		&& \quad\text{in $D$,}
\end{alignedat}
\label{E:TRSIMP}
\end{equation}
with $T>0$ and $D \subset \R^d$ a bounded Lipschitz domain, for 
\begin{equation}
	C \in \L^\infty(Q),
	\quad
	F \in \L^2(Q),
	\quad
	\bar{U} \in \L^2(D).
\label{E:DATA}
\end{equation}
Formally \eqref{E:TRSIMP} can be interpreted as a family of ODEs and solved
explicitly by
\begin{equation}
\begin{gathered}
	U(t,x) = \Bigg( \bar{U}(x) + \int_0^t F(s,x) 
		\exp\big( I(s,x) \big) \,ds \Bigg) \exp\big( -I(t,x) \big),
\\
	I(t,x) := \int_0^t C(r,x) \,dr
\end{gathered}
\label{E:EXPLICIT}
\end{equation}
for $t\in [0,T]$, with $x\in D$ \emph{fixed}. We are going to rephrase the
problem in variational form, which is more suitable for the transport equation
\eqref{E:TREQN3}. 

%========== REMARK
\begin{remark}\label{R:NONPOS}
We may assume that $C(t,x)\GS 0$ for a.e.\ $(t,x)$. Indeed if $U$ is a solution
of \eqref{E:TRSIMP}, then the rescaled function $\tilde{U}(t,x) := \exp(-\lambda
t) U(t,x)$ satisfies
\[
\begin{alignedat}{2}
	\partial_t \tilde{U} + (C+\lambda) \tilde{U} & = \tilde{F}
		&& \quad\text{in $Q$,}
\\
	\tilde{U}(0,\cdot) & = \bar{U}
		&& \quad\text{in $D$,}
\end{alignedat}
\]
where $\tilde{F}(t,x) := \exp(-\lambda t) F(t,x)$ for a.e.\ $(t,x)$. Then
$C+\lambda \GS 0$ if $\lambda \GS \|C\|_{\L^\infty(Q)}$. In particular, we can
consider the case that $C$ vanishes.
\end{remark}

For any $U,V \in \C(\bar{Q}) \cap \C^1(Q)$ we define
\begin{equation}
	\BOP_\circ U := \partial_t U + CU
	\quad\text{and}\quad
	\BOP_\circ^* V := -\partial_t V + CV,
\label{E:BOPN}
\end{equation}
with $\BOP_\circ^*$ the formal adjoint of $\BOP_\circ$. Defining the spaces
\[
	\C^1_\pm(Q) := \big\{ U \in \C(\bar{Q})\cap\C^1(Q) \colon
		U(s_\pm,\cdot) = 0 \big\}
\]
with $s_- := 0$, $s_+ := T$, and using Green's Theorem, we have that
\[
	\langle\BOP_\circ U, V\rangle_{\L^2(Q)}
		= \langle U, \BOP_\circ^* V\rangle_{\L^2(Q)}
	\quad\text{for all $U\in \C^1_-(Q)$ and $V\in \C^1_+(Q)$,}
\]
where $\langle\cdot,\cdot\rangle_{\L^2(Q)}$ is the $\L^2(Q)$-inner product. One
can show that the graph norm
\begin{equation}
	\|V\|_* := \|\BOP_\circ^* V\|_{\L^2(Q)}
\label{E:GRAH}
\end{equation}
is indeed a norm on $\C^1_+(Q)$ since 
\[
	\|V\|_{\L^2(Q)} \LS 2T \|\BOP_\circ^* V\|_{\L^2(Q)}
	\quad\text{for all $V\in\C^1_+(Q)$;}
\]
see Proposition~2.2(i) in \cite{BrunkenSmetanaUrban2019}. (Note that the flow
associated to the constant velocity field in \eqref{E:TRSIMP} is trivially
$Q$-filling, as defined there.) Let $\Y$ be the closure of $\C^1_+(Q)$ with
respect to \eqref{E:GRAH}, which is a Hilbert space with inner product/norm
\[
	\langle V,W\rangle_\Y := \langle \BOP^*V, \BOP^*W\rangle_{\L^2(Q)}
	\quad\text{and}\quad
	\|V\|_\Y := \|\BOP^*V\|_{\L^2(Q)}
\]
for all $V, W \in \Y$. Here $\BOP^* \colon \Y \longrightarrow \L^2(Q)$ denotes
the continuous extension of $\BOP_\circ^*$ from $\C^1_+(Q)$ to $\Y$. We define
$\BOP \colon \L^2(Q) \longrightarrow \Y'$ by duality as $\BOP := (\BOP^*)^*$.

We can now present the variational formulation of \eqref{E:TRSIMP}, which will
be the basis of the Petrov-Galerkin approximation in Section~\ref{SS:PGA}. Let 
\begin{equation}
	b(U,V) := \langle U,\BOP^* V\rangle_{\L^2(Q)}
	\quad\text{for all $U\in \L^2(Q)$ and $V\in\Y$.}
\label{E:BILINEAR}
\end{equation}
Moreover, define the linear functional
\begin{equation}
	\ell(V) := \langle F,V\rangle_{\L^2(Q)} + \int_D \bar{U}(x) V_0(x) \,dx
	\quad\text{for all $V\in\Y$,}
\label{E:RHS}
\end{equation}
where $V_0$ is the strong $\L^2(D)$-trace of the test function $V$ on the
boundary $\{0\}\times D$. The existence of traces was proved in Proposition~2.4
in \cite{DahmenHuangSchwabWelper2011}, for instance.

%========== THEOREM
\begin{theorem}\label{T:EXI}
For functions $C\GS 0, F$ and initial data $\bar{U}$ as in \eqref{E:DATA}, let
the bilinear form $b$ and the functional $\ell$ be defined by \eqref{E:BILINEAR}
and \eqref{E:RHS}, respectively. Then 
\begin{equation}
	\sup_{U\in\L^2(Q)} \sup_{V\in\Y} 
			\frac{b(U,V)}{\|U\|_{\L^2(Q)} \|V\|_\Y}
		= \inf_{U\in\L^2(Q)} \sup_{V\in\Y} 
			\frac{b(U,V)}{\|U\|_{\L^2(Q)} \|V\|_\Y}
		= 1,
\label{E:OPTIM}
\end{equation}
so that the continuity and $\inf$-$\sup$-conditions for the form $b$ are
satisfied with optimal constant $1$. This is equivalent to the following
stability estimates:
\begin{gather*}
	\|\BOP\|_{\LINEAR(\L^2(Q), \Y')} 
		= \|\BOP^{-1}\|_{\LINEAR(\Y', \L^2(Q))} = 1, 
\\
	\|\BOP^*\|_{\LINEAR(\Y, \L^2(Q))} 
		= \|(\BOP^*)^{-1}\|_{\LINEAR(\L^2(Q), \Y)} = 1.
\end{gather*}
Here $\LINEAR$ denotes spaces of bounded linear maps, and $\Y'$ is the dual of
$\Y$. In particular, there exists a unique $U\in \L^2(Q)$ that solves equation
\eqref{E:TRSIMP} in the sense that
\begin{equation}
	b(U,V) = \ell(V)
	\quad\text{for all $V\in\Y$,}
\label{E:VARIATIONAL}
\end{equation}
and the stability estimate $\|U\|_{\L^2(Q)} \LS \|\ell\|_{\Y'}$ holds. 

Moreover, for any $W \in \L^2(Q)$ we have that
\begin{equation}
	\|U-W\|_{\L^2(Q)} = \|R_W\|_{\Y'},
\label{E:ERROR}
\end{equation}
with \emph{residual} $R_W \in \Y'$ defined for all $V\in\Y$ as
\[
	R_W(V) := b(U-W,V) = \ell(V)-b(W,V).
\]
Therefore the approximation error can be controled in terms of the residual.
\end{theorem}

%========== PROOF
\begin{proof}
This is \cite{BrunkenSmetanaUrban2019} Theorem~2.5, which is based on
\cite{DahmenHuangSchwabWelper2011}. Then \eqref{E:ERROR} follows from
\[
	\|U-W\|_{\L^2(Q)} \left\{ \genfrac{}{}{0pt}{}{\LS}{\GS} \right\}
		\sup_{V\in\Y} \frac{b(U-W,V)}{\|V\|_\Y}
			= \sup_{V\in\Y} \frac{R_W(V)}{\|V\|_\Y} = \|R_W\|_{\Y'},
\]
which is a consequence of the continuity and the $\inf$-$\sup$-condition
\eqref{E:OPTIM}.
\end{proof}

%========== REMARK
\begin{remark}
Given assumptions \eqref{E:DATA} on the data, the map $U$ in \eqref{E:EXPLICIT}
is well-defined for $t\in [0,T]$ and a.e.\ $x\in D$. In particular, $U$ belongs
to $\L^2(Q)$, it is differentiable with respect to $t$, and $\partial_t U =
-CU+F$ belongs to $\L^2(Q)$ as well. Using integration by parts, one finds that
$U$ satisfies the variational formulation \eqref{E:VARIATIONAL}. Since
variational solutions are unique, the $U$ in \eqref{E:EXPLICIT} coincides with
the unique solution to \eqref{E:VARIATIONAL}.
\end{remark}

As explained in the Introduction, we want to solve transport equations of the
form \eqref{E:TREQN3} by decoupling geometry and transport. 

%========== DEFINITION
\begin{definition}\label{D:EXB}
Let $I := [0,T]$ for $T>0$, $D\subset\R^d$ a bounded Lipschitz domain, and $Q :=
(0,T)\times D$. Consider $\VELO \in \L^\infty(\bar{Q};\R^d)$ and $\FLOW \colon
\bar{Q} \longrightarrow \bar{D}$ such that
\begin{enumerate}
	
\item For all $t\in I$, the map $x \mapsto \FLOW(t, x)$ is Borel measurable.
		
\item For all $x \in \bar{D}$, the map $t \mapsto \FLOW(t, x)$ is a Filippov
solution of 
\[
	\dot{\gamma}(t) = \VELO\big( t,\gamma(t) \big)
	\quad
	\text{for $t\in I$,}
	\quad
	\gamma(0) = x.
\]

\item The flow lines are forward untangled: for $s\in I$ and $x_1, x_2 \in
\bar{D}$
\[
	\text{$\FLOW(s,x_1) = \FLOW(s,x_2)$}
	\quad\Longrightarrow\quad
	\text{$\FLOW(t,x_1) = \FLOW(t,x_2)$ for all $s\LS t\LS T$.}
\]

\end{enumerate}
For all $t\in I$, we define a density $\RHO(t,\cdot)$ as in \eqref{E:RHOT} with
initial measure $\bar{\RHO} := \LEB^d\llcorner D$, and the space-time measure
$\sigma(dz, dt) := \RHO(t,dz) \,dt$ on $\bar{Q}$. Consider now
\begin{equation}
	c \in \L^\infty(\bar{Q}, \sigma),
	\quad
	f \in \L^2(\bar{Q}, \sigma),
	\quad
	\bar{u} \in \L^2(D).
\label{E:DATA3}
\end{equation}
We say that $u\in \L^2(\bar{Q}, \sigma)$ is a \emph{flow-solution} of
\begin{equation}
\begin{alignedat}{2}
	\partial_t u + \VELO\cdot\nabla u + cu & = f
		&& \quad\text{in $\bar{Q}$,}
\\
	u(0,\cdot) & = \bar{u}
		&& \quad\text{in $D$,}
\end{alignedat}
\label{E:TRSE}
\end{equation}
if the following is true: for all $\varphi \in \C(\bar{Q})$ it holds
\begin{equation}
	\iint_{\bar{Q}} u(t,z) \varphi(t,z) \,\RHO(t,dz) \,dt
		= \iint_Q U(t,x) \varphi\big( t,\FLOW(t,x) \big) \,dx \,dt,
\label{E:DEFU}
\end{equation}
where $U \in \L^2(Q)$ is a variational solution of \eqref{E:TRSIMP} (satisfying
\eqref{E:VARIATIONAL}) with
\begin{equation}
	C(t,x) := c\big( t,\FLOW(t,x) \big),
	\quad
	F(t,x) := f\big( t,\FLOW(t,x) \big),
	\quad
	\bar{U}(x) := \bar{u}(x)
\label{E:DATA2}
\end{equation}
for $(t,x) \in (0,T)\times D$. We assume that $c \GS 0$ $\sigma$-a.e.; cf.\
Remark~\ref{R:NONPOS}.
\end{definition}

As was shown in Theorem~\ref{T:SELECTION}, a flow $\FLOW$ as in
Definition~\ref{D:EXB} does exist, but may be non-unique. Once the flow has been
selected we have well-posedness:

%========== THEOREM
\begin{theorem}
Under the assumptions listed in Definition~\ref{D:EXB}, there exists a unique
flow-solution of the initial value problem \eqref{E:TRSE} of the transport
equation.
\end{theorem}

%========== PROOF
\begin{proof}
Our assumptions \eqref{E:DATA3} imply that Theorem~\ref{T:EXI} can be applied.
Indeed recall that the push-forward formula \eqref{E:RHOT} implies the identity
\begin{equation}
	\iint_Q g\big( t,\FLOW(t,x) \big) \,dx \,dt 
		= \iint_{\bar{Q}} g(t,z) \,\RHO(t,dz) \,dt 
\label{E:PUSHL}
\end{equation}
for all $g \in \L^\infty(I\times\Omega)$ with bounded support. From this, we get
\[
	\iint_Q |F(t,x)|^2 \,dx \,dt = \iint_{\bar{Q}} |f(t,z)|^2 \,\RHO(t,dz) \,dt
\]
(using a suitable approximation argument), which is finite. In a similar way,
one can check that the other functions in \eqref{E:DATA2} satisfy
\eqref{E:DATA}, as required.
\end{proof}

Notice that the function
\begin{equation}
	V(t,x) := \zeta\big( t,\FLOW(t,x) \big)
	\quad\text{for $(t,x) \in Q$}
\label{E:ZETAT}
\end{equation}
belongs to the test space $\Y$, for any $\zeta \in \C(\bar{Q}) \cap \C^1(Q)$.
Indeed since the function $V$ is bounded and $Q$ has finite Lebesgue measure,
and since
\begin{equation}
	\partial_t V(t,x) = \partial_t\zeta\big( t,\FLOW(t,x) \big)
		+ \VELO\big( t,\FLOW(t,x) \big) \cdot \nabla\zeta\big( t,\FLOW(t,x) \big),
\label{E:PARTIAL}
\end{equation}
by definition of the flow $\FLOW$, the claim follows from the boundedness of
$\VELO$. We can use this test function in the variational formulation
\eqref{E:VARIATIONAL} and obtain
\begin{align}
	& \iint_Q U(t,x) \Big\{ -\partial_t\zeta\big( t,\FLOW(t,x) \big)
		- \VELO\big( t,\FLOW(t,x) \big) \cdot \nabla\zeta\big( t,\FLOW(t,x) \big)
\nonumber\\
	& \hspace*{18em}
		+ C(t,x) \zeta\big( t,\FLOW(t,x) \big) \Big\} \,dx \,dt
\nonumber\\
	& \qquad
			= \iint_Q F(t,x) \zeta\big( t,\FLOW(t,x) \big) \,dx \,dt
			+ \int_D \bar{U}(x) \zeta(x) \,dx.
\label{E:PASSTO}
\end{align}
Because of \eqref{E:DATA2} and \eqref{E:PUSHL} (and another approximation
argument), this gives
\begin{align*}
	& \iint_{\bar{Q}} u(t,z) \Big\{ -\partial_t\zeta(t,z)
		- \VELO(t,z) \cdot \nabla\zeta(t,z)
		+ c(t,z) \zeta(t,z) \Big\} \,\RHO(t,dz) \,dt
\\
	& \qquad
		= \iint_{\bar{Q}} f(t,z) \zeta(t,z) \,\RHO(t,dz) \,dt
			+ \int_D \bar{u}(z) \zeta(z) \,\bar{\RHO}(dz),
\end{align*}
which is the weak formulation of the initial value problem
\[
\begin{alignedat}{2}
	\partial_t (\RHO u) + \nabla\cdot(\RHO u\VELO) + c\RHO u & = \RHO f
		&& \quad\text{in $\bar{Q}$,}
\\
	(\RHO u)(0,\cdot) & = \bar{\RHO}\bar{u}
		&& \quad\text{in $D$.}
\end{alignedat}
\]
Since $\RHO$ is a weak solution of the continuity equation
\[
\begin{alignedat}{2}
	\partial_t \RHO + \nabla\cdot(\RHO \VELO) & = 0
		&& \quad\text{in $\bar{Q}$,}
\\
	\RHO(0,\cdot) & = \bar{\RHO}
		&& \quad\text{in $D$,}
\end{alignedat}
\]
we formally get \eqref{E:TRSE} using the product rule. Therefore the notion of
flow-solution is indeed a suitable solution concept for the transport equation
\eqref{E:TRSE}.

%========== REMARK
\begin{remark}
If there exist $s\in I$ and $x_1,x_2 \in \bar{D}$ such that $\FLOW(s,x_1) = z =
\FLOW(s,x_2)$, then the solution $u$ at $(s,z)$ is found as the superposition of
the values of $u$ along all incoming flow lines, because of \eqref{E:DEFU}. This
is similar to the sticky particle dynamics for pressureless gas dynamics; see
\cite{BrenierGangboSavareWestdickenberg2013}. In particular, if all flow lines
starting from a set with positive Lebesgue measure merge at $(s,z)$, then the
density $\RHO$ becomes singular, forming a Dirac measure at $z$, for instance.
Since the flow lines are assumed to be forward untangled, however, we have that
$C(t,x_1) = C(t,x_2)$ and $F(t,x_1) = F(t,x_2)$ for all $s\LS t\LS T$; recall
\eqref{E:DATA2}. The ODEs \eqref{E:TRSIMP} at the points $x_1$ and $x_2$,
respectively, are therefore identical after the merging; see
Figure~\ref{F:UNTANGLED}.
%
% \begin{figure}[t]
% \centering
% %
% \begin{minipage}[b]{0.495\textwidth}
% \centering
% \includesvg[extractformat=pdf,height=0.5\textwidth]{flowlines.svg}
% \end{minipage}
% %
% \hfill
% %
% \begin{minipage}[b]{0.495\textwidth}
% \centering
% \includesvg[extractformat=pdf,height=0.5\textwidth]{odes.svg}
% \end{minipage}
% %
% \caption{Merging of flow lines and untangled ODEs}
% \label{F:UNTANGLED}
% \end{figure}
%
\begin{figure}[t]
\centering
\begin{minipage}[b]{0.495\textwidth}
\centering
\includegraphics[height=0.5\textwidth]{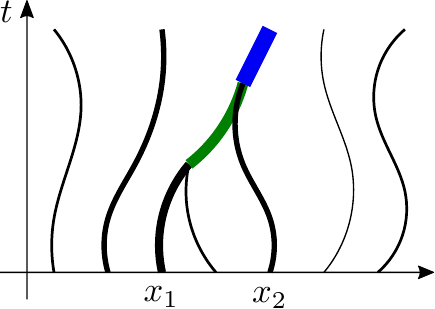}
\end{minipage}
\hfill
\begin{minipage}[b]{0.495\textwidth}
\centering
\includegraphics[height=0.5\textwidth]{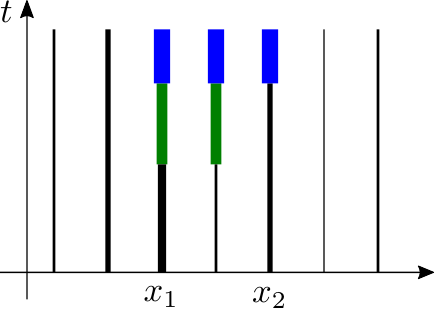}
\end{minipage}
\caption{Trajectories in physical and parameter space}
\label{F:UNTANGLED}
\end{figure}
\end{remark}

We have the following compactness result for flow solutions of \eqref{E:TRSE}.

%========== THEOREM
\begin{theorem}\label{T:COMPACT}
Let $I := [0,T]$ for $T>0$, $D \subset\R^d$ a bounded Lipschitz domain, and $Q
:= (0,T)\times D$. Consider a sequence of velocity fields $\VELO_n \in
\L^\infty(\bar{Q};\R^d)$ with corresponding flows $\FLOW_n \colon \bar{Q}
\longrightarrow \bar{D}$ whose flow lines are forward untangled. For all $t\in
I$, we define a density $\RHO_n(t,\cdot)$ as in \eqref{E:RHOT} with initial data
$\bar{\RHO} := \LEB^d \llcorner D$, and the space-time measure $\sigma_n(dz,dt)
:= \RHO_n(t,dz) \,dt$ in $\bar{Q}$. Consider
\[
	c_n \in \L^\infty(\bar{Q}, \sigma_n),
	\quad
	f_n \in \L^2(\bar{Q}, \sigma_n),
	\quad
	\bar{u}_n \in \L^2(D),
\]
and define initial data $\bar{U}_n(x) := \bar{u}_n(x)$ and functions
\begin{equation}
	C_n(t,x) := c_n\big( t,\FLOW_n(t,x) \big),
	\quad
	F_n(t,x) := f_n\big( t,\FLOW_n(t,x) \big)
\label{E:NEWDATA}
\end{equation}
for all $(t,x) \in (0,T)\times D$. We assume that $c_n \GS 0$ $\sigma_n$-a.e.;
cf.\ Remark~\ref{R:NONPOS}.

Suppose that $u_n\in \L^2(\bar{Q}, \sigma_n)$ are flow-solutions of
\begin{equation}
\begin{alignedat}{2}
	\partial_t u_n + \VELO_n\cdot\nabla u_n + c_n u_n & = f_n
		&& \quad\text{in $\bar{Q}$,}
\\
	u_n(0,\cdot) & = \bar{u}_n
		&& \quad\text{in $D$,}
\end{alignedat}
\label{E:APFLO}
\end{equation}
where $\BOP_n$ and $\BOP_n^*$ and the associated bilinear forms $b_n$ and
right-hand sides $\ell_n$ are defined as above, using the indexed quantities
$C_n$, $F_n$, and $\bar{U}_n$; see \eqref{E:BOPN}, \eqref{E:BILINEAR}, and
\eqref{E:RHS}. Let $U_n \in \L^2(Q)$ be the generator of the flow-solution $u_n$
as in \eqref{E:DEFU}.

We now make the following assumptions:
\begin{enumerate}

\item There exist $\VELO \in \L^\infty(\bar{Q}; \R^d)$ and an associated flow
$\FLOW \colon \bar{Q} \longrightarrow \bar{D}$ with
\begin{equation}
	\begin{aligned}
		\VELO_n & \longrightarrow \VELO
\\
		\FLOW_n & \longrightarrow \FLOW
	\end{aligned}
	\Bigg\rbrace
	\quad\text{pointwise a.e.\ in $\bar{Q}$.}
\label{E:CONVF}
\end{equation}
The flow lines of $\FLOW$ are forward untangled.

\item There exists $\bar{u} \in \L^2(D)$ such that 
\begin{equation}
	\bar{u}_n \WEAK \bar{u}
	 \quad\text{weakly in $\L^2(D)$.}
\label{E:WINIT}
\end{equation}

\item There exist $c \in \L^\infty(\bar{Q}, \sigma)$ and $f \in \L^2(\bar{Q},
\sigma)$ such that
\begin{alignat}{2}
	C_n & \WEAK C &&
	\quad\text{pointwise a.e.\ in $Q$,}
\label{E:CONVC}\\
	F_n & \WEAK F &&
	\quad\text{weakly in $\L^2(Q)$,}
\label{E:WRHS}
\end{alignat}
with $C_n, F_n$ and $C, F$ defined by \eqref{E:NEWDATA} and \eqref{E:DATA2},
respectively, and with density $\RHO$ and space-time measure $\sigma$ as in
Definition~\ref{D:EXB}.

\end{enumerate}
There exists a subsequence (not relabeled) and $U\in \L^2(Q)$ with $U_n \WEAK U$
weakly in $\L^2(Q)$ such that $u\in \L^2(\bar{Q}, \sigma)$ defined by
\eqref{E:DEFU} is a flow-solution of \eqref{E:TRSE}.
\end{theorem}

%========== REMARK
\begin{remark}
A typical situation to which Theorem~\ref{T:COMPACT} can be applied, arises when
the velocity $\VELO$ and the data $c, f$, and $\bar{u}$ in \eqref{E:TRSE} are
approximated by piecewise constant or linear finite elements over a suitable
discretization of $\Omega$. The approximate velocity field $\VELO_n$ induces an
approximate flow $\FLOW_n$, which in turn determines approximate flow solutions
$u_n$; see \eqref{E:APFLO}. Then Theorem~\ref{T:COMPACT} establishes the
convergence of $u_n$ towards a flow solution of \eqref{E:TRSE}. The crucial
assumption here is \eqref{E:CONVF}.
\end{remark}

%========== PROOF
\begin{proof}
The weak convergences \eqref{E:WINIT} and \eqref{E:WRHS} imply uniform
boundedness of the linear forms $\ell_n$ in $\Y'$. Because of the uniform
stability estimate $\|U_n\|_{\L^2(Q)} \LS \|\ell_n\|_{\Y'}$ (see
Theorem~\ref{T:EXI}), it follows that the sequence of $U_n$ is uniformly
bounded, thus weakly precompact. We can therefore extract a subsequence (not
relabeled) such that $U_n \WEAK U$ for some $U \in \L^2(Q)$. Notice that for any
$\varphi \in \C(\bar{Q})$, the map
\[
	(t,x) \mapsto \varphi\big( t, \FLOW_n(t,x) \big)
	\quad\text{for $(t,x) \in \bar{Q}$}
\]
is uniformly bounded in $\L^\infty(\bar{Q})$ and converges pointwise a.e. This
implies that
\[
	\sigma_n \WEAK \sigma
	\quad\text{weak* in the sense of measures,}
\]
for some $\sigma \in \M(\bar{Q})$ that can be disintegrated in the form
$\sigma(dz, dt) =: \RHO(t,dz) \,dt$. Moreover, the measure $\RHO$ satisfies
\eqref{E:PUSHL}. Similarly, we have the convergence
\[
	\sigma_n u_n \longrightarrow \BM
	\quad\text{weak* in the sense of measures,}
\]
for some finite measure $\BM \in \M(\bar{Q})$. Because of lower semicontinuity
of 
\[
	(\sigma,\BM) \mapsto \sup_{\zeta \in \C(\bar{Q})} \iint_{\bar{Q}} \bigg(
		\zeta \,d\BM - \frac{1}{2} |\zeta|^2 \,d\sigma \bigg) 
\]
(see Section~2.6 in \cite{AmbrosioFuscoPallara2000}), it follows that
$\BM(dz,dt) =: u(t,z) \,\RHO(t,dz) \,dt$ with 
\[
	u(t,\cdot) \in \L^2\big( \bar{D}, \RHO(t,\cdot) \big)
	\quad\text{for all $t\in I$.}
\]
In particular, the functions $U \in \L^2(Q)$ and $u \in \L^2(\bar{Q}, \sigma)$
are related by \eqref{E:DEFU}.

It remains to check that $U$ is a solution of the variational problem
\eqref{E:VARIATIONAL}. This follows from passing to the limit in the integrals
defining the bilinear forms $b_n$ and the right-hand sides $\ell_n$. For
instance, for any $V \in \Y$ the maps
\[
	(t,x) \mapsto C_n(t,x) V(t,x)
	\quad\text{for $(t,x) \in Q$}
\]
converge \emph{strongly in $\L^2(Q)$} because the first factor remains bounded
and converges pointwise a.e., by assumption \eqref{E:CONVC}. It follows that
\[
	\iint_Q U_n(t,x) C_n(t,x) V(t,x) \,dx \,dt
		\longrightarrow 
			\iint_Q U(t,x) C(t,x) V(t,x) \,dx \,dt
\]
because $U_n \WEAK U$ weakly in $\L^2(Q)$. In a similar way, we find that
\begin{align*}
	\iint_Q U_n(t,x) \partial_t V(t,x) \,dx \,dt
		& \longrightarrow 
			\iint_Q U(t,x) \partial_t V(t,x) \,dx \,dt
\\
	\iint_Q F_n(t,x) V(t,x) \,dx \,dt
		& \longrightarrow 
			\iint_Q F(t,x) V(t,x) \,dx \,dt
\\
	\int_D \bar{U}_n(x) V_0(x) \,dx
		& \longrightarrow 
			\int_D \bar{U}(x) V_0(x) \,dx,			
\end{align*}
where $V_0$ is the strong $\L^2(D)$-trace of $V\in \Y$ on the
boundary $\{0\}\times D$. Thus
\begin{equation}
	b(U,V) 
		= \lim_{n\to\infty} b_n(U_n, V)
		= \lim_{n\to\infty} \ell_n(V) = \ell(V)
\label{E:LIMIT}
\end{equation}
for all $V\in\Y$. This concludes the proof of the theorem.
\end{proof}

%========== REMARK
\begin{remark}
A sufficient condition for \eqref{E:CONVC} and \eqref{E:WRHS} is that $c_n$ and
$f_n$
\begin{itemize}

\item are Borel measurable, 

\item uniformly bounded in $\L^\infty(\bar{Q})$,

\item converge pointwise a.e.\ to $c$ and $f$, respectively.

\end{itemize}
In this case, for any $V \in \Y$ the maps
\[
	(t,x) \mapsto C_n(t,x) V(t,x) = c_n\big( t,\FLOW_n(t,x) \big) V(t,x)
	\quad\text{for $(t,x) \in Q$}
\]
converge \emph{strongly in $\L^2(Q)$} because the first factor remains bounded
and converges pointwise a.e., by assumptions \eqref{E:CONVF} and
\eqref{E:CONVC}. The same applies to 
\[
	(t,x) \mapsto F_n(t,x) = f_n\big( t,\FLOW_n(t,x) \big)
	\quad\text{for $(t,x) \in Q$.}
\]
We then pass to the limit in the integrals
\begin{align*}
	\iint_Q U_n(t,x) c_n\big( t,\FLOW(t,x) \big) V(t,x) \,dx \,dt
		& \longrightarrow 
			\iint_Q U(t,x) c\big( t,\FLOW(t,x) \big) V(t,x) \,dx \,dt,
\\
	\iint_Q f_n\big( t,\FLOW(t,x) \big) V(t,x) \,dx \,dt
		& \longrightarrow 
			\iint_Q f\big( t,\FLOW(t,x) \big) V(t,x) \,dx \,dt.
\end{align*}
We can argue as above to obtain \eqref{E:LIMIT}.

If there exists a constant $C>0$ such that the densities $\RHO_n$ satisfy the
inequalities \eqref{E:NIC} below uniformly in $n$, then a sufficient condition
for \eqref{E:WRHS} is for $f_n$ to \emph{converge strongly in $\L^2(Q)$}. In
this case, integration with respect to the space-time measure $\sigma_n$ is
equivalent to integration against the $(d+1)$-dimensional Lebesgue measure on
$Q$. Indeed, let us define the truncation operator $T_R \colon \R
\longrightarrow \R$ as
\[
	T_R(s) := \begin{cases}
		s & \text{if $|s| \LS N$}
\\
		Ns/|s| & \text{otherwise}
	\end{cases}
	\quad\text{for all $s\in\R$.}
\]
We can then estimate
\begin{multline*}
	\iint_Q \big| (1-T_R)\big( F_n(t,x) \big) \big|^2 \,dx \,dt
		= \iint_Q \Big| (1-T_R)\Big( f_n\big( t,\FLOW_n(t,x) \big) \Big) 
			\Big|^2 \,dx \,dt
\\
		= \iint_Q \big| (1-T_R)\big( f_n(t,z) \big) \big|^2 \,\RHO_n(t,dz) \,dt
		\LS C \iint_Q \big| (1-T_R)\big( f_n(t,z) \big) \big|^2 \,dz \,dt,
\end{multline*}
which is arbitrarily small uniformly in $n$ if $R>0$ is large enough. For the
function $T_R(F_n)$ we can argue as above. This gives $F_n \longrightarrow F$
strongly in $\L^2(Q)$. 
\end{remark}

%========== REMARK
\begin{remark}
As in the proof of Theorem~\ref{T:COMPACT}, we find that the maps
\[
	V(t,x) := \zeta\big( t,\FLOW_n(t,x) \big)
	\quad\text{for $(t,x)\in Q$}
\]
converge \emph{strongly in $\Y$} to the function $V$ defined in \eqref{E:ZETAT},
for any $\zeta \in \C(\bar{Q})\cap\C^1(Q)$. Indeed note that the derivative
$\partial_t V_n$ is of the form \eqref{E:PARTIAL}. We can then combine the
uniform boundedness of $(\partial_t\zeta, \nabla\zeta)$ and the velocity fields
$\VELO_n$, with the pointwise convergence \eqref{E:CONVF} to obtain strong
convergence in $\L^2(Q)$ for both $V$ and its partial derivative $\partial_t V$.
This enables us to pass to the limit in the analogue of \eqref{E:PASSTO}.
\end{remark}

%%%%%%%%%%%%%%%%%%%%%%%%%%%%%%%%%%%%%%%%%%%%%%%%%%%%%%%%%%%%%%%%%%%%%%%%%%%%%%%%
%%%%%%%%%%%%%%%%%%%%%%%%%%%%%%%%%%%%%%%%%%%%%%%%%%%%%%%%%% Stability of Flow Map

\subsection{Stability of Flow Map}\label{SS:SOFM}

We are interested in the situation where the sequence of flows $\FLOW_n$ in
Theorem~\ref{T:COMPACT} is generated by an approximation $\VELO_n$ of the
velocity field, e.g., in a numerical method. A recent result by Bianchini and
Bonicatto \cite{BianchiniGloyer2011} provides strong convergence in $\L^1(Q)$ if
the sequence of (approximate) velocity fields is uniformly of bounded variation
and nearly incompressible.

%========== DEFINITION
\begin{definition}
We say that a velocity $\VELO \in \L^\infty(Q; \R^d)$ is \emph{nearly
incompressible} if there exists a function $\RHO \in \L^\infty(Q)$ and a
constant $C>0$ such that 
\begin{equation}
	C^{-1} \LS \RHO(t,x) \LS C
	\quad\text{for a.e.\ $(t,x) \in Q$,}
\label{E:NIC}
\end{equation}
where $\RHO$ is a distributional solution of the continuity equation
\eqref{E:CONTEQN}. We will assume in the following that the initial data is
given by the constant function: $\RHO(0,x) = 1$ for all $x\in\Omega$, in which
case we will refer to $\RHO$ as a \emph{density} associated to $\VELO$. We will
also say that $\RHO$ is nearly incompressible if \eqref{E:NIC} holds.
\end{definition}
	
%========== REMARK
\begin{remark}
A sufficient condition for near incompressibility is that $\nabla\cdot\VELO$ is
bounded. Indeed the continuity equation can (formally) be rewritten as
\[
	\partial_t \RHO + \VELO\cdot\nabla \RHO = -\RHO \nabla\cdot\VELO
	\quad\text{in $Q$,}
\]
so that the change of $\log\RHO$ along the integral curves of $\VELO$ is given
by $-\nabla\cdot\VELO$. In general, near incompressibility is strictly weaker
since it allows for certain discontinuities in the velocity field, as is
relevant for hyperbolic conservation laws, for instance.
\end{remark}

%========== REMARK
\begin{remark}
Near incompressibility encodes important information about the flow: On the one
hand, the lower bound in \eqref{E:NIC} (absence of vacuum) implies that there
are no sets of positive Lebesgue measure that are not reached by the flow. This
is a weak form of surjectivity. On the other hand, the upper bound in
\eqref{E:NIC} (absence of concentrations) guarantees that no sets of positive
Lebesgue measure are mapped to small sets such as a single point. This is a weak
form of injectivity.
\end{remark}

We can now state the main stability result for approximate flow maps.

%========== THEOREM
\begin{theorem}\label{T:BRESSAN}
Let $I := [0,T]$ for $T>0$, $D \subset\R^d$ a bounded Lipschitz domain, and $Q
:= (0,T)\times D$. Consider a velocity field $\VELO \in \L^1((0,T); \BVS(D;
\R^d))$ bounded and nearly incompressible. Then there exists a unique regular
Lagrangian flow (RLF) for $\VELO$, i.e., a flow $\FLOW \colon \bar{Q}
\longrightarrow \bar{D}$ associated to $\VELO$ such that for a.e.\ $t\in(0,T)$
\[
	\text{$A\subset \bar{D}$ Borel with $|A|=0$}
	\quad\Longrightarrow\quad
	|\{ x\in\bar{D} \colon \FLOW(t,x)\in A\}| = 0.
\]
Here $|\cdot|$ denotes the $d$-dimensional Lebesgue measure. 

Moreover, consider a sequence of velocities $\VELO_n$ with the following
properties:
\begin{itemize}

\item For all $n\in\N$, we have $\VELO_n \in \L^1((0,T); \BVS(D; \R^d))$.

\item The quantity $\|\VELO_n\|_{\L^\infty(Q)}$ is bounded uniformly in $n$ and 
\[
	\VELO_n \longrightarrow \VELO
	\quad\text{strongly in $\L^1(Q)$.}
\]

\item There exists a constant $C>0$ such that for all $n\in\N$
\[
	C^{-1} \LS \RHO_n(t,x) \LS C
	\quad\text{for a.e.\ $(t,x) \in Q$,}
\]
where $\RHO_n$ is defined by \eqref{E:RHOT}, with $\FLOW_n$ the RLF generated by
$\VELO_n$.

\end{itemize}
Then $\FLOW_n \longrightarrow \FLOW$ strongly in $\L^1(Q)$.
\end{theorem}

The statement of Theorem~\ref{T:BRESSAN} was first formulated in
\cite{Bressan2003} and is nowadays known as Bressan's conjecture. It was proved
very recently by Bianchini and Bonicatto \cite{BianchiniBonicatto2017}. The
proof establishes the existence of a set of \emph{untangled flow lines} that
partition the space-time domain and induce the disintegration of the continuity
equation into one-dimensional transport problems along integral curves. One can
then show that the flow has the renormalization property, from which Bressan's
conjecture can be deduced; see Theorem~3.22 in \cite{DeLellis2007} for details.
It would be very interesting to make the stability/convergence of the
approximate flows $\FLOW_n$ more quantitative, along the lines of
\cite{Seis2019}, for example; see also \cite{CrippaDeLellis2008} for additional
information.

%========== REMARK
\begin{remark}
Another stability result for approximate flows was derived in
\cite{BianchiniGloyer2011} for Filippov solutions for \emph{monotone} velocity
fields, which satisfy the inequality
\begin{equation}
	\langle \VELO(t,y)-\VELO(t,x), y-x\rangle \LS \alpha(t) |y-x|^2
\label{E:MONO}
\end{equation}
for a.e.\ $t\in I$ and $x,y\in \R^d$, where $\alpha \in \L^1(I)$. Rescaling
time, one can reduce the problem to the case of zero right-hand side in
\eqref{E:MONO}; cf.\ Remark~\ref{R:NONPOS}. The condition \eqref{E:MONO} is
sufficient to ensure uniqueness of Filippov solutions; see \cite{Filippov1964}.
In the context of transport equations it has been studied  in
\cites{PoupaudRascle1997,BouchutJamesMancini2005}. Note that \eqref{E:MONO}
corresponds to a compressive situation where $\nabla\cdot\VELO$ may become a
negative measure. 
\end{remark}

%%%%%%%%%%%%%%%%%%%%%%%%%%%%%%%%%%%%%%%%%%%%%%%%%%%%%%%%%%%%%%%%%%%%%%%%%%%%%%%%
%%%%%%%%%%%%%%%%%%%%%%%%%%%%%%%%%%%%%%%%%%%%%%%%%% Petrov Galerkin Approximation

\subsection{Petrov-Galerkin Approximation}\label{SS:PGA}

To approximate the variational problem \eqref{E:VARIATIONAL}, an optimally
stable Petrov-Galerkin method can be used, which we describe now. We follow the
presentation in \cite{BrunkenSmetanaUrban2019}. Let $\U \subset \L^2(Q)$ and $\V
\subset \Y$ be suitable trial and test spaces (e.g., finite dimensional finite
element spaces). Suppose that $\U, \V$ are Hilbert (sub)spaces. Then we aim to
find $U\in\U$ such that
\begin{equation}
	b(U,V) = f(V)
	\quad\text{for all $V\in\V$.}
\label{E:PG}
\end{equation}
Arguing as for Theorem~\ref{T:EXI}, one can show that \eqref{E:PG} has a unique
solution if
\begin{equation}
	\gamma 
		\GS \sup_{U\in\U} \sup_{V\in\V} 
			\frac{b(U,V)}{\|U\|_{\L^2(Q)} \|V\|_\Y}
		\GS \inf_{U\in\U} \sup_{V\in\V} 
			\frac{b(U,V)}{\|U\|_{\L^2(Q)} \|V\|_\Y} 
		\GS \beta
\label{E:DISC}
\end{equation}
for constants $\beta>0$ and $\gamma<\infty$. Since $\U\subset \L^2(Q)$ and
$\V\subset \Y$, by assumption, the upper bound in \eqref{E:DISC} holds with
$\gamma=1$ because of \eqref{E:OPTIM}. 

Let $U\in \L^2(Q)$ and $U_\U \in \U$ be unique solutions of
\eqref{E:VARIATIONAL} and \eqref{E:PG}, respectively. Since $\V \subset \Y$ we
obtain the Galerkin orthogonality
\begin{equation}
	b(U-U_\U, V) = 0
	\quad\text{for all $V\in\V$.}
\label{E:ORTHO}
\end{equation}
The $\inf$-$\sup$-condition \eqref{E:DISC} now implies that, for any $W \in \U$,
\begin{align*}
	\beta \|U_\U-W\|_{\L^2(Q)} 
		& \LS \sup_{V\in\V} \frac{b(U_\U-W,V)}{\|V\|_\Y}
\\
		& = \sup_{V\in\V} \frac{b(U-W,V)}{\|V\|_\Y}
			\LS \|U-W\|_{\L^2(Q)}.
\end{align*}
We have used \eqref{E:ORTHO} and the continuity of $b$ with $\gamma=1$. Then
\begin{align*}
	\|U-U_\U\|_{\L^2(Q)}
		& \LS \|U-W\|_{\L^2(Q)} + \|U_\U-W\|_{\L^2(Q)}
\\
		& \LS (1+\beta^{-1}) \|U-W\|_{\L^2(Q)},
\end{align*}
and so the error can be estimated by the best approximation of $U$ in $\U$:
\[
	\|U-U_\U\|_{\L^2(Q)} \LS (1+\beta^{-1})
		\inf_{W \in \U} \|U-W\|_{\L^2(Q)}.
\]
Note the dependence on the $\inf$-$\sup$-constant $\beta>0$ from \eqref{E:DISC}.
As the trial space $\U$ approaches $\L^2(Q)$, the approximation $U_\U$ converges
to the exact solution $U$.

In order to compute $\beta$, for any $U \in \U$ we want to solve
\begin{equation}
	\sup_{V \in \V} \frac{b(U, V)}{\|U\|_{\L^2(Q)} \|V\|_{\V}}.
\label{E:MAXI}
\end{equation}
We rewrite \eqref{E:MAXI} as a \emph{constrained} maximization problem
\[
	\sup_{V \in \V} \frac{b(U, V)}{\|U\|_{\L^2(Q)} \|V\|_{\V}}
		= \frac{\sup\{ b(U, V) \colon V\in\V, \|V\|_{\V} = 1 \}}
			{\|U\|_{\L^2(Q)}},
\]
which can be solved using Lagrange's method: We consider the Lagrangian
\begin{equation}
	L_U(V,\lambda) := b(U,V) + \lambda \big( 1-\|V\|_{\V}^2 \big)
\label{E:LAGR}
\end{equation}
with $V \in \V$ and $\lambda \in \R$. We denote by $(V_\circ, \lambda_\circ)$
the critical points of \eqref{E:LAGR}, which are characterized by the fact that
the first variations of \eqref{E:LAGR} vanish. Thus
\begin{equation}
	b(U, W) - 2\lambda_\circ \langle V_\circ,W \rangle_{\V} = 0
	\quad\text{for all $W\in\V$,}
\label{E:ELG}
\end{equation}
and the normalization condition $\|V_\circ\|_{\V}^2 = 1$ holds. Using $W =
V_\circ$ in \eqref{E:ELG}, we find that $2\lambda_\circ = b(U, V_\circ)$.
Defining $V_U := 2\lambda_\circ V_\circ$, we have the identity
\begin{equation}
	b(U,W) = \langle V_U,W \rangle_{\V}
	\quad\text{for all $W\in\V$,}
\label{E:BUW}
\end{equation}
which shows that $V_U$ is actually the unique element in $\V$ representing the
bounded linear map $V \mapsto b(U,V)$ on the Hilbert space $\V$, in the sense of
the \emph{Riesz representation} theorem. Note that, by definition and
\eqref{E:BUW}, we have the identities
\begin{equation}
	\langle U, \BOP^*V_U \rangle_{\L^2(Q)} 
		= b(U,V_U)
		= \| V_U \|_{\V}^2
		= \langle \BOP^*V_U, \BOP^*V_U \rangle_{\L^2(Q)}.
\label{E:IDEN}
\end{equation}

It is convenient to make the choice that
\begin{equation}
	\U := \BOP^*(\V),
\label{E:CONVENIENT}
\end{equation}
in which case \eqref{E:IDEN} implies that $U = \BOP^*V_U$, or $V_U =
(\BOP^*)^{-1}U$ because the linear map $\BOP^*$ is bounded invertible; see
Theorem~\ref{T:EXI} above. We call the map $\tau \colon \U \longrightarrow \V$,
defined by $\tau(U) := V_U$ for all $U\in\U$, the \emph{trial-to-test map}. Then
\[
	\sup_{V \in \V} \frac{b(U, V)}{\|U\|_{\L^2(Q)} \|V\|_{\V}}
		= \frac{b(U, V_U)}{\|U\|_{\L^2(Q)} \|V_U\|_{\V}}
		= \frac{\langle U, \BOP^* (\BOP^*)^{-1} U \rangle_{\L^2(Q)}}
			{\|U\|_{\L^2(Q)} \|\BOP^* (\BOP^*)^{-1} U\|_{\L^2(Q)}}
		= 1
\]
for all $U \in \U$. Taking the $\inf/\sup$ in $U$, we obtain the continuity
estimate and the $\inf$-$\sup$-condition with constant $1$. Therefore the method
is optimally stable.

\medskip

In order to realize an optimally conditioned and thus optimally stable
Petrov-Galerkin method that is also computationally feasible, it was proposed in
\cite{BrunkenSmetanaUrban2019} to first choose a conformal finite-dimensional
test space $\V_h \subset \Y$, which depends on some discretization parameter
$h>0$, and then to define the trial space through \eqref{E:CONVENIENT}. Possible
choices for $\V_h$ include spaces of sufficiently smooth functions such as
splines or broken spaces of finite elements on suitable triangulations of the
domain. It is then proved that the $\inf$-$\sup$-property holds uniformly in
$h$. The authors also consider the dependence of the variational problem on a
compact set of parameters and prove compactness of the solution set (analogous
to our Theorem~\ref{T:COMPACT}).

A more sophisticated framework was developed in
\cite{BroersenDahmenStevenson2018} utilizing a \emph{discontinuous}
Petrov-Galerkin approximation. In this case, the unknown is a pair of $U \in
\L^2(Q)$ as above and an additional function that captures the jumps across the
boundaries of the triangulation cells. This is motivated by the fact that for
generic $\L^2(Q)$-functions the trace onto hyperplanes may not be defined. The
authors prove near optimal stability of this approximation with discretized
data. 

%========== REMARK
\begin{remark}
The discussion above suggests a numerical method for solving the initial value
problem \eqref{E:TRSE} that consists of two steps:
\begin{enumerate}

\item In physical space-time approximate the velocity field $\VELO$ using finite
elements on a suitable triangulation. Compute the corresponding flow $\FLOW$.

\item Solve the variational formulation of the transport equation $\partial_t U
+ CU = F$, with functions $C$ and $F$ defined by \eqref{E:DATA2}.

\end{enumerate}
The triangulation of the first step induces a decompositon of $Q := (0,T) \times
D$ for the second step, as cells are mapped by the flow $\FLOW$. It is possible
to iterate the two steps adaptively, taking into account the error estimates in
\cite{BroersenDahmenStevenson2018}, for instance. The implementation of this
approach will be considered in a future publication.
\end{remark}

%%%%%%%%%%%%%%%%%%%%%%%%%%%%%%%%%%%%%%%%%%%%%%%%%%%%%%%%%%%%%%%%%%%%%%%%%%%%%%%%
%%%%%%%%%%%%%%%%%%%%%%%%%%%%%%%%%%%%%%%%%%%%%%%%%%%%%%%%%%%%%%%%%%%%%%%%%%%%%%%%
%%%%%%%%%%%%%%%%%%%%%%%%%%%%%%%%%%%%%%%%%%%%%%%%%%%%%%%%%%%%%%%%%%%%%%%%%%%%%%%%
%%%%%%%%%%%%%%%%%%%%%%%%%%%%%%%%%%%%%%%%%%%%%%%%%%%%%%%%%%%%%%%%%%%%%%%%%%%%%%%%
%%%%%%%%%%%%%%%%%%%%%%%%%%%%%%%%%%%%%%%%%%%%%%%%%%%%%%%%%%%%%%% Acknowledgements

\section*{Acknowledgements}

The project has received funding from the European Union's Horizon 2020 research
and innovation programme under the Marie Sk\l{}odowska-Curie grant agreement
No.\ 642768. Financial support has also been provided by the German Research
Foundation (DFG) in the framework of the Collaborative Research Centre
Transregio 40 (SFB-TR40) and the Research Training Group GRK 2326/1.

The authors would like to thank the anonymous referee for her/his careful and
thorough reading of the manuscript and for pointing out reference
\cite{Gusev2018}.

%%%%%%%%%%%%%%%%%%%%%%%%%%%%%%%%%%%%%%%%%%%%%%%%%%%%%%%%%%%%%%%%%%%%%%%%%%%%%%%%
%%%%%%%%%%%%%%%%%%%%%%%%%%%%%%%%%%%%%%%%%%%%%%%%%%%%%%%%%%%%%%%%%%%%%%%%%%%%%%%%
%%%%%%%%%%%%%%%%%%%%%%%%%%%%%%%%%%%%%%%%%%%%%%%%%%%%%%%%%%%%%%%%%%%%%%%%%%%%%%%%
%%%%%%%%%%%%%%%%%%%%%%%%%%%%%%%%%%%%%%%%%%%%%%%%%%%%%%%%%%%%%%%%%%%%%%%%%%%%%%%%
%%%%%%%%%%%%%%%%%%%%%%%%%%%%%%%%%%%%%%%%%%%%%%%%%%%%%%%%%%%%%%%%%%% Bibliography

% biblatex
\printbibliography

% bibtex
% \bibliography{Transport}{}
% \bibliographystyle{alpha}

\end{document}